\documentclass[moor]{informs1}              % for a regular run

\usepackage{natbib}
 \NatBibNumeric
 \def\bibfont{\small}%
 \bibpunct[, ]{[}{]}{,}{n}{}{,}%
 
 \usepackage{graphicx}
 \usepackage{url,color}
 \usepackage{subfigure}
 \usepackage{amsfonts,mathrsfs}
 \usepackage{amssymb,amsmath}
 \usepackage{verbatim}
 \usepackage{acronym}
 \usepackage{multirow}
 \usepackage{array}
 \usepackage[utf8]{inputenc}
 \usepackage[english]{babel}
 \usepackage[mathscr]{euscript}
 \usepackage{bbold}
 \usepackage{bbm}
 
 \newcommand{\rset}{\mathbb{R}}
 \newcommand{\Hf}{H_{f,U_{k}}}
 \newcommand{\HfU}{H_{f,U}}
 \newcommand{\Hpsi}{H_{\psi, U_{k}}}
 
 \newcommand{\Hfmax}{H_{f,\max}}
 \newcommand{\Hpsimax}{H_{\psi,\max}}
 \newcommand{\HF}{H_{F,U_{k}}}
 \newcommand{\HFmax}{H_{F,\max}}
 
 \newcommand{\inas}{\overset{a.s.}{\to}}
 \newcommand{\red}{\textcolor{black}}
\usepackage[colorlinks=true,breaklinks=true,bookmarks=true,urlcolor=blue,
     citecolor=blue,linkcolor=blue,bookmarksopen=false,draft=false]{hyperref}

\def\EMAIL#1{\href{mailto:#1}{#1}}% When hyperref is used, otherwise outcomment 
\def\URL#1{\href{#1}{#1}}         % When hyperref is used, otherwise outcomment 

\TheoremsNumberedThrough     % Preferred (Theorem 1, Lemma 1, Theorem 2)
\EquationsNumberedThrough    % Default: (1), (2), ...
\begin{document}
%%%%%%%%%%%%%%%%

\TITLE{Efficiency of stochastic coordinate  proximal gradient  methods on nonseparable composite  optimization}

\ARTICLEAUTHORS{%
\AUTHOR{Ion  Necoara}
\AFF{Automatic Control and Systems
	Engineering Department, University Politehnica Bucharest, 060042
	Bucharest, Romania  and  Gheorghe Mihoc-Caius Iacob  Institute of Mathematical Statistics and Applied Mathematics of the Romanian Academy, 050711 Bucharest, Romania, \EMAIL{ion.necoara@upb.ro.} \URL{}}
\AUTHOR{Flavia Chorobura}
\AFF{Automatic Control and Systems
	Engineering Department, University Politehnica Bucharest, 060042
	Bucharest, Romania, \EMAIL{flavia.chorobura@stud.acs.upb.ro.} \URL{}}
} 

\ABSTRACT{%
This paper deals with composite  optimization problems having  the  objective function formed as the sum of two  terms, one has Lipschitz continuous gradient along random subspaces and may be nonconvex and the second  term is simple and differentiable, but  possibly nonconvex and \textit{nonseparable}.  Under these settings we design a stochastic  coordinate proximal gradient method which takes into account the nonseparable composite form of the objective function. This algorithm achieves scalability by constructing at each iteration a local approximation model of the whole nonseparable  objective function along  a random  subspace with user-determined dimension.  We outline efficient techniques for selecting the random subspace, yielding an implementation that has  low cost per-iteration  while also achieving fast convergence rates.  We present a probabilistic worst-case complexity analysis for our stochastic  coordinate proximal gradient method in convex and nonconvex settings, in particular we prove high-probability bounds on the number of iterations before a given optimality is achieved.  Extensive numerical results also  confirm the  efficiency of our  algorithm. 
%To the best of our knowledge, this work is the first proposing a pure stochastic  coordinate descent algorithm which is supported by global efficiency estimates for general  nonseparable  composite optimization problems.  
}

% Sample
%\KEYWORDS{deterministic inventory theory; infinite linear programming duality; 
%  existence of optimal policies; semi-Markov decision process; cyclic schedule}
%\MSCCLASS{Primary: 90B05; secondary: 90C40, 90C90}
%\ORMSCLASS{Primary: Inventory/production: deterministic multi-item;
%  secondary: dynamic programming/optimal control: deterministic 
%  semi-Markov; programming: infinite dimensional}
%\HISTORY{Received November 20, 2003; revised March 8, 2004, and March 26, 2004.}

\MSCCLASS{90C25, 90C06, 65K05.}
%\ORMSCLASS{Primary: ; secondary: }
\HISTORY{January 2023}

\maketitle

%%%%%%%%%%%%%%%%%%%%%%%%%%%%%%%%%%%%%%%%%%%%%

\section{Introduction.}\label{intro} %%1.
\noindent This paper develops stochastic (block) coordinate  proximal gradient methods for solving composite optimization problems of the form: 
\begin{align}
	\label{eq:prob}
	F^* = \min_{x \in \mathbb{R}^n} F(x) := f(x) + \psi(x) ,
\end{align}
where $f: \mathbb{R}^n  \to \mathbb{R}$  is smooth (possibly  nonconvex)   and $\psi : \mathbb{R}^n  \to {\mathbb{R}}$ is simple and twice differentiable  (possibly nonseparable and  nonconvex).  Optimization problems having this composite structure
permit to handle general coupling functions   (e.g.,  $\psi(x) = \| A x\|^p$, with $A$ linear operator and $p \geq 2$) and arise in many applications such as  distributed control, signal  processing, truss topology design,  machine learning, traffic equilibrium problems, network flow problems and other areas \cite{Ber:99,LuFreNes:18,Mit:97,NecCli:13}.   In particular, this model is suitable for applications where for the second term, although differentiable, the computations of the gradient  or of the coordinatewise Lipschitz constants over a bounded set along a subspace are expensive or even impossible; on the other hand the computation of the prox for the second term  along a subspace is easy. Despite the nonsmooth property of the sum, such problems, both in convex and nonconvex cases, can be solved by full gradient or Newton methods with the efficiency typical for the smooth part of the objective \cite{Nes:18,NesPol:06}.  {However, for large-scale problems,  the usual methods based on full gradient/prox  and on Hessian  computations are prohibitive. In this case, a reasonable approach to solve such problems   is to use  (block) coordinate descent methods, see~\cite{Ber:99}.  }

\medskip 

\noindent \textit{State of the art}.  The variants of  (block) coordinate descent algorithms differ from each other in the way  we define the local approximation function over which we optimize  and the criterion of choosing at each iteration the subspace over which we minimize this local approximation function.  For updating one (block) variable, while keeping the other variables fixed,  two basic choices for the local approximation function  are usually considered: (i) exact approximation function, leading to \textit{coordinate minimization methods} \cite{Ber:99,GriSci:00},  and (ii)  quadratic approximation function, leading to \textit{coordinate (proximal) gradient}  \textit{descent   methods}  \cite{BecTet:13,FerRic:15,LuXia:14,Nes:10,NecCli:16,NecTak:20,RicTak:11,TseYun:09}.  Furthermore, three classical criteria for choosing the coordinate search  used often in these algorithms are the greedy, the cyclic and the random coordinate search, which significantly differs by the amount of computations required to choose the appropriate index. For cyclic coordinate search rates of convergence have been given only recently \cite{BecTet:13}. Convergence rates for coordinate descent methods based on the Gauss-Southwell rule  were given e.g.  in \cite{TseYun:09}. Another interesting approach is based on stochastic coordinate descent, where the coordinate search is random. Recent complexity results on stochastic coordinate descent methods were obtained in \cite{Nes:10}  for smooth functions and the extension to composite  functions  were given e.g., in \cite{FerRic:15,LuXia:14,NecCli:16,RicTak:11}.  However, these  papers studied optimization models where the second term, usually assumed nonsmooth, is separable,  i.e., $\psi(x) =\sum_{i=1}^n \psi_i(x_i)$, with $x_i$ is the $i$th  component of $x$.  In the sequel, we discuss papers that consider the case $\psi$  nonseparable and explain the main diferences with our present work. 

\medskip 

\noindent \textit{Previous work}.  From our knowledge there exist very few studies  on  coordinate descent methods when the second term of the composite objective function in  \eqref{eq:prob} is nonseparable. For example,  \cite{Nec:13,NecPat:14,NecTak:20,TseYun:09} considers  the composite optimization problem \eqref{eq:prob}  with the additional  linear constraints $Ax=b$.  Hence,   nonseparability comes from the linear constraints $Ax=b$, as $\psi$ is  assumed convex and  separable (possibly nonsmooth). Under these settings \cite{Nec:13,NecPat:14,NecTak:20,TseYun:09}  propose coordinate gradient descent methods where at the current feasible  point $x$ one needs to solve a subproblem over a  subspace generated by the matrix $U \in \mathbb{R}^{n \times p}$ using  a part of the gradient of $f$, $\nabla f(x)$, i.e.:
\[  \min_{d \in \mathbb{R}^p} f(x) + \langle  U^T \nabla f(x),d \rangle + \frac{1}{2} d^T H_U d  \quad \text{s.t.} \quad AUd=0,  \]
where $H_U$ is an appropriate positive definite  matrix  and then update $x^+ = x + Ud$.  The matrix $U$ is chosen according to some greedy rule or random. For this type of algorithms  sublinear rates are derived in the (non)convex case and linear convergence is obtained for  strongly convex objective.   Further, for general  $\psi$ (possibly nonseparable and nonsmooth)  \cite{GriIut:21,HanKon:18,LatThe:21} consider proximal coordinate descent  methods  of the form: 
\begin{equation*}
	\bar{x}^{+} \in  \text{prox}_{\alpha \psi } (\mathcal{C}(x - \alpha \nabla f(x))), 
\end{equation*}  
\noindent where $\mathcal{C}( \cdot)$   is a correction map  corresponding to the chosen random subspace at the current iteration in \cite{GriIut:21,HanKon:18} and is the identity map in \cite{LatThe:21}.  Moreover,  \cite{GriIut:21,HanKon:18} assume $\psi$ convex and update $x^{+} = \bar{x}^{+}$, while in \cite{LatThe:21} $\psi$ is possibly nonconvex and updates   $x_{i}^{+} = \bar{x}_{i}$  for all $i \in \mathcal{I}  \subseteq [n]$ and keeps the rest of the components  unchanged.  Note that  \cite{GriIut:21,HanKon:18}   use only  a sketch of  the gradient   $\nabla f(x)$ on the selected subspace, while in  \cite{LatThe:21}  $f$ is assumed separable.  { Other work for $\psi$ nonseparable is  \cite{FerBia:19}, where  structured problems are considered of the form: 
	\vspace*{-0.2cm}
	\begin{equation*}
		\min_{x \in \mathbb{R}^n} f(x) + \psi(x) + h(Ax),
		\vspace*{-0.2cm}
	\end{equation*}
	\noindent with $f$,$\psi$ and $h$ convex functions, such that 
	$\psi$ and $h$ are possible nonsmooth and  $f$ is smooth. For solving this problem, a randomized coordinate descent primal-dual algorithm is proposed where at each iteration  a block of components of the prox of the function    $\psi$,  $\text{prox}_{\alpha \psi }$, must be computed.  %Under these settings, in \cite{FerBia:19}, a sublinear convergence rate is derived in the general convex case and linear rate when the objective function is strongly convex. 
	\red{ Since in these papers \cite{FerBia:19,GriIut:21,HanKon:18,LatThe:21},  one needs to compute at each iteration a block of components of the full prox of the nonseparable function $\psi$,  $\text{prox}_{\alpha \psi }$,  this can be can done efficiently  when this prox  can be evaluated easily given that only a block of coordinates are modified in the prior iteration.} }
%However,  in these papers  one needs to compute the full prox of function $\psi$,  $\text{prox}_{\alpha \psi }$, which can be expensive,  unless $\psi$ has some particular structure that allows efficient computation of individual blocks of prox without the need of full operations.  For the  algorithms in \cite{GriIut:21,HanKon:18}  linear convergence is derived, provided that the smooth term $f$ is  strongly convex and $\psi$ is convex. Linear convergence results were also obtained in \cite{LatThe:21}  when the objective function satisfies the Kurdyka-Lojasiewicz  condition.  
For the  algorithms in \cite{FerBia:19,GriIut:21,HanKon:18}  linear convergence rate is derived, provided that the objective function  is  strongly convex. Linear convergence rate results were also obtained in \cite{LatThe:21}  when the objective function satisfies the Kurdyka-Lojasiewicz  condition.  Recently, \cite{AbeBec:21} considers  problem \eqref{eq:prob},  where the function  $f$ is assumed quadratic and convex,  while $\psi$ convex   function (possibly nonseparable and nonsmooth).  Under these settings,  \cite{AbeBec:21}  combines  the forward-backward envelope (to  smooth the original problem)  with an accelerated coordinate gradient descent method and derives sublinear rates for the proposed scheme. { This method also makes sense when the full prox can be computed efficiently under coordinate descent updates. \red{The main difference between  our work and  \cite{FerBia:19,GriIut:21,HanKon:18,LatThe:21,AbeBec:21} is that  we consider a prox along subspace, while in the other papers  one needs to compute a block of components of the full prox.} }	
 Recently,  in \cite{CarRob:21} a trust region algorithm is proposed based on iterative minimization within random subspaces for solving smooth nonconvex unconstrained problems, i.e. $\psi=0$.  Using a probabilistic argument from \cite{ChuLu:02},  paper  \cite{CarRob:21} derives a high-probability bound on the number of iterations such that the norm of the gradient is less than some given accuracy.  
 %Recall  that in the literature, 
%the convergence of  coordinate descent methods is guaranteed only when $f$ is smooth and  $\psi$ is separable or $\psi \equiv 0$, while 
\red{Recall that, for  $\psi$ nonseparable, coordinate descent methods may not converge (see e.g., the counterexamples in \cite{FriHasHofTib} for  nonseparable nondifferentiable convex problems   and in \cite{Wri:15} for nonseparable nonconvex problems, even in the differentiable case). These results motivate us to consider $\psi$ twice differentiable.}
 %\textit{ Therefore, there is  an open problem in the optimization literature whether it is possible to design a pure coordinate proximal gradient scheme for the  composite problem \eqref{eq:prob} by imposing more properties on $\psi$,  which is supported by global efficiency estimates.}   By  pure coordinate proximal gradient scheme we mean that at each iteration the method sketches  the gradient   $\nabla f(x)$  and computes the  prox of $\psi$  along some subspace.   In this paper we solve this problem, i.e., we design for the composite problem \eqref{eq:prob},  having $\psi$ twice differentiable,  a stochastic coordinate proximal gradient method that requires sketching  the gradient   $\nabla f(x)$ on some random subspace and computing the  prox of $\psi$ also along this subspace. 

\medskip 

\noindent \textit{Contributions}.  This paper deals with the  composite  problem \eqref{eq:prob} having  the  objective function formed as a sum of two  terms, one has Lipschitz continuous gradient along random subspaces and may be nonconvex   and the second  term is  simple and twice differentiable, but  possibly \textit{nonconvex} and \textit{nonseparable}.  Under these settings we design a stochastic  coordinate proximal gradient  method which takes into account the nonseparable composite form of the objective, i.e., at each iteration the method sketches  the gradient   $\nabla f(x)$  and computes the  prox of $\psi$ only along some subspace chosen in a stochastic manner\footnote{\red{To make the distinction between our approach and  random coordinate descent schemes, we use the term "stochastic",  since  in our method the subspace over which we minimize and the underlying probability distribution are general, while in random coordinate descent only blocks of coordinates are chosen based on a discrete probability distribution.}}.  Our algorithm achieves scalability by constructing at each iteration a local approximation model of the   objective  along  a random  subspace with user-determined dimension. {We outline efficient techniques for selecting the random subspace, yielding an implementation that has  low cost per-iteration   (linear in  problem dimension), while also achieving fast~convergence.  } 

\medskip 

\noindent The main challenge in analyzing the convergence of our scheme  is the fact that the cost cannot be used as a Lyapunov function,  unless some stochastic embedding of the gradient is imposed.    Another challenge arises  when analyzing convergence of  a stochastic algorithm under the local  Kurdyka-Lojasiewicz  (KL) condition.  We overcome this problem using Egorov's theorem, which allows  us to pass from almost sure convergence to uniform convergence.  Hence, combining our stochastic  embedding with  nontrivial probabilistic arguments, we present a probabilistic worst-case complexity analysis for our stochastic  coordinate proximal gradient method in convex and nonconvex settings,  {where in particular we prove high-probability bounds on the number of iterations before a given optimality is achieved.} The convergence rates in probability in both convex and nonconvex settings are summarized in the table below. In particular, sublinear  rates are derived for nonconvex and convex problems. Improved rates are given under the KL inequality (in the nonconvex case) and uniform convexity (in the convex case).    
\begin{small}
	\begin{center}
		\begin{tabular}{  |p{2.5cm}| p{4.5cm}| p{0.4cm}| p{4.5cm}| p{0.4cm}|}
			\hline
			\multirow{2}{*}{} & \multicolumn{2}{|c|}{convex} & \multicolumn{2}{|c|}{nonconvex}\\
			\cline{2-5}
			& \multicolumn{1}{|c|}{convergence rates} & \multicolumn{1}{|c|}{Theorem} & \multicolumn{1}{|c|}{convergence rates} & \multicolumn{1}{|c|}{Theorem}\\
			\hline
			%& &  & &  \\
			sublinear  with high probability   &  \begin{small}{\(\displaystyle F(x_k) - F^* \!\leq\! \mathcal{O}\left(k^{-1}\right) \) } \end{small}  &  \ref{the:conv}  & \begin{small} $\min\limits_{i=1:k} \! \nabla F(x_{i}) \!\leq\! \mathcal{O} \! \left(\!k^{-\frac{1}{2}} \!\right)$ \end{small}  & \ref{theo1}  \\
			&   \begin{small} $F(x_k) \to F^*$ sublinearly depending on  uniform convexity parameter \end{small}  &   \ref{theo:KL1}   &   \begin{small} $F(x_k) \to F_*$ sublinearly depending on KL parameter \end{small} & \ref{theo:KL1} \\
			\hline
			& &  & &  \\
			linear with high probability  &  \begin{small}  $F(x_k) - F^* \!\leq\! \mathcal{O}\left( c^{k} \right)$ (strong convexity)  \end{small}  &  \ref{theo:KL1}  &     \begin{small}  $F(x_k) \to F_*$ linearly or superlinearly  depending on KL parameter  \end{small}  & \ref{theo:KL1} \\
			\hline
		\end{tabular}
	\end{center}
\end{small}

\medskip

%\noindent  Besides providing a general framework for the analysis of  stochastic coordinate proximal gradient schemes, our results  resolve an open problem in the optimization literature related to the  convergence estimates for coordinate descent algorithms on \textit{nonseparable} (non)convex composite problems.   Moreover,
\noindent Note that, our  convergence results are also of  interest when $f \equiv 0$ and $\psi$ nonseparable (in this case our algorithm can be viewed as a proximal regularization of a multi-block Gauss-Seidel method).  There are very few results ensuring that the iterates of a multi-block Gauss-Seidel method converges to a global minimizer, even for strictly convex functions, e.g., \cite{Ber:99,GriSci:00} present only asymptotic convergence results. An immediate byproduct  of our convergence analysis leads however  to novel convergence results for the popular multi-block Gauss-Seidel algorithm  in  both convex and nonconvex settings with very general sampling strategies.  An important property of our method in this case is that the convergence holds for arbitrarily small regularization parameters.   The numerical results also  confirm the  efficiency of our  algorithm.

\medskip 

\noindent \textit{Content}. The paper is organized as follows. In Section 2 we present some preliminary results,  while in Section 3 we define our problem of interest and a stochastic  coordinate proximal gradient algorithm.  The convergence rates in probability are derived in Section 4 for the nonconvex case and in Section 5 for the convex case. In the last section we provide detailed numerical simulations.

%%%%%%%%%%%%%%%%%%

\section{Preliminaries}
\noindent  In this section we present our basic assumptions for composite problem  \eqref{eq:prob},   some definitions and  preliminary results.  We consider the following problem setting.  Let  $\mathbb{S}  = \{ \mathbb{S}_\omega \}_\omega $ be  family  of (linear) subspaces of $\mathbb{R}^{n}$. Intuitively, this set represents the directions that will be favored by the random descent.  We  assume that each subspace $\mathbb{S}_\omega$  is generated by the columns of a certain matrix $U_\omega \in \mathbb{R}^{n \times p}$ and denote the family of such  matrices by $\mathcal{U} = \{ U_\omega \}_\omega \subset \mathbb{R}^{n\times p}$.   For simplicity of the notation we consider the same dimension  $p$ for each subspace, although our results also hold for subspaces having different dimensions. Throughout the paper we assume that the set $\mathcal{U}$ is bounded (we can relax this assumption considering that this set is bounded with high probability).  Note that stochasticity enters in our algorithmic framework through a user-defined distribution on $\mathcal{U}$ describing the ensemble of random matrices $U  \in    \mathbb{R}^{n\times p}$, notation $U \sim \mathcal{U}$. The basic idea of our algorithmic framework consists of a given  $x \in \rset^n$, a sample  matrix $U \sim \mathcal{U}$ and a basic update of the form: 
$$x^+ = x+ Ud. $$ 

\begin{comment}
	Clearly, one can choose
	a distribution  which will not guarantee convergence to stationary/optimal points of \eqref{eq:prob}. Therefore, we need to impose some minimal necessary conditions for such a scheme to be well-defined. In particular, in order to avoid trivial updates, we need to choose $U \sim \mathcal{U}$  such that: 
	\[   \left( x + \text{Range}(U) \right)  \cap  \text{dom} \,\psi \not = \{x\}.  \]
	Moreover, in order to reach a stationary/optimal point, we naturally assume that our subspaces (matrices) in the family can generate $\text{dom} \,\psi$, that is:
	\begin{equation}  
		\label{eq:condU}
		\text{dom} \,\psi \subseteq    \text{Span} \left(  \bigcup\limits_{U \sim \mathcal{U}}  \left( x + \text{Range}(U) \right)  \cap  \text{dom} \,\psi    \right)  \quad \forall x \in \text{dom} \,\psi.  
	\end{equation}
	For example, if $\text{dom} \,\psi = \rset^n$, then \eqref{eq:condU}  requires that our family of matrices spans the whole space. On the other hand, if $\psi$ is the indicator function of the set $\{x: Ax=b\}$, where $A \in \rset^{m \times n}$, \eqref{eq:condU} requires (see also \cite{NecTak:20}):
	\[   \text{Ker}(A) \subseteq     \text{Span} \left(  \bigcup\limits_{U \sim \mathcal{U}}   \text{Range}(U)   \cap  \text{Ker}(A)    \right).    \]     
\end{comment}

\noindent Throughout the paper the following assumptions will be valid:
\begin{assumption}
	\label{ass2}
	For the composite optimization problem \eqref{eq:prob} the following hold:
	\begin{itemize}
		\item []A.1: Given a family of matrices $\mathcal{U} \subset \mathbb{R}^{n\times p}$, then  for any $U \in \mathcal{U}$  the function  $f$ is $L_U$ smooth along $U$ (i.e., the gradient is Lipschitz continuous in the subspace generated by $U$): 
		%with modulus $L_U$):
		\begin{align}
			\label{lip1}
			\|U^{T}(\nabla f(x + U h) - \nabla f(x)) \| \leq L_{U} \|h\| \quad  \forall h \in \mathbb{R}^{p}.
		\end{align}
		\item []A.2:  The function $\psi$ is  simple { (i.e., the proximal operator along the subspace generated by $U$ can be computed easily)}  and  twice continuously differentiable (possibly nonseparable and nonconvex). 
		\item[]A.3: A solution exists for \eqref{eq:prob}  (hence, the optimal value $F^* >- \infty$). 
	\end{itemize}
\end{assumption}	

%\noindent  Unlike typical cases analyzed in the literature where either $f$ or $\psi$  is assumed  separable (convex) \cite{BecTet:13,FerRic:15,LuXia:14,LatThe:21,Nes:10,NecCli:16,RicTak:11}, we consider the case where  both terms are  \textit{nonseparable}  (\textit{nonconvex}), respectively.  

\noindent An immediate consequence of Assumption \ref{ass2} [A.1] is the following \cite{Nes:18}: 
\begin{lemma}
	\label{lemma:1}
	If  Assumption \ref{ass2} [A1] holds, then we have the relation:
	\begin{align}
		\label{lip2}
		|f(x+U h) - f(x) - \langle U^{T} \nabla f(x), h \rangle| \leq \frac{L_{U}}{2} \|h\|^2 \quad \forall h \in  \mathbb{R}^{p}.
	\end{align}
\end{lemma}	
\proof{Proof}
	See Appendix  for a proof. \Halmos
\endproof

\medskip

\noindent Note that when  $f$ has gradient Lipschitz with global Lipschitz constant $L$, we also  have Assumption \ref{ass2} [A.1] valid with $L_U = L \|U\|^{2}$ since: 
\begin{eqnarray*}
	\|U^{T}(\nabla f(x + U h) - \nabla f(x)) \| \leq \|U\|\|\nabla f(x + U h) - \nabla f(x)\| 
	\leq L \|U\|^{2} \|h\|. 
\end{eqnarray*}

\noindent Another example when Assumption \ref{ass2} [A.1] holds is given in  next lemma \cite{Nes:18}: 
\begin{lemma}
	\label{lemma:2}
	If $f$ is twice differentiable, then  \eqref{lip1} holds  if and only if the following relation holds:
	\begin{equation}
		\|U^{T}\nabla^2 f(x) U \| \leq L_{U} \ \ \forall x \in \mathbb{R}^{n}. \label{eq:16}
	\end{equation} 
\end{lemma}
\proof{Proof}
	See Appendix for a proof. \Halmos
\endproof

\medskip

\noindent In our convergence analysis we also use functions that   are uniform convex  \cite{Nes:19} or satisfy the Kurdyka-Lojasiewicz  property \cite{BolDan:07}. Note that these properties are usually defined for general functions (possibly nondifferentiable). Below, we adapt these definitions to the differentiable case, since in this paper we consider only differentiable objective functions.

\medskip

\begin{definition}
	A differentiable function $\psi$ is uniformly convex of order $r \geq 2$, if there exists $\sigma_r > 0$ such that the following relation holds:  
	\begin{align}
		\label{eq:unifConv}
		\psi(y) \geq \psi(x) + \langle \nabla \psi(x), y - x \rangle + \frac{\sigma_r}{r} \|y - x\|^{r} \quad \forall x, y \in \mathbb{R}^{n}. 
	\end{align}
\end{definition}

\noindent Note that for $r = 2$ in \eqref{eq:unifConv}  we recover  the usual  definition of  a strongly convex function.  One important class of uniformly convex functions is  $ \psi(x) = \frac{1}{r} \|  x - \bar{x}\|^{r}$,  
with $\bar{x}$ fixed,  see e.g.,  \cite{Nes:19}.  Minimizing both sides of  \eqref{eq:unifConv}, we also get:
\begin{align}
	\label{eq:probConv}
	\psi^{*} =\min _{y \in \rset^n} \psi(y) & \geq \min _{y \in \rset^n}\left\{\psi(x)+\left\langle \nabla \psi(x), y-x\right\rangle+\frac{\sigma_{r}}{r}\|y-x\|^{r}\right\} \nonumber\\
	&=\psi(x)-\frac{(r-1)}{r}\left(\frac{1}{\sigma_{r}}\right)^{\frac{1}{r-1}}\left\|\nabla \psi(x)\right\|_{*}^{\frac{r}{r-1}} \quad \forall x \in  \rset^n.  
\end{align}

\noindent For nonconvex functions the  usual notion replacing uniform convexity \eqref{eq:probConv} is the Kurdyka-Lojasiewicz  property, which captures a broad spectrum of the local geometry that a nonconvex function can have \cite{BolDan:07}. 

\medskip

\begin{definition}
	\label{def:kl}
	\noindent A differentiable  function $\psi$ satisfies the \textit{Kurdyka-Lojasiewicz (KL)} property on a given compact set $\mathcal{X}$ on which $\psi$ takes a constant value $\psi_*$ if there exist $\gamma, \epsilon >0$ such that   one~has:
	\begin{equation*}
		\kappa' (\psi(x) - \psi_*)\| \nabla \psi(x) \|  \geq 1  \quad   \forall x\!:  \text{dist}(x, \mathcal{X}) \leq \gamma, \;  \psi_* < \psi(x) < \psi_* + \epsilon,  
	\end{equation*}
	where $\kappa: [0,\epsilon] \to \mathbb{R}$ is a concave differentiable function satisfying $\kappa(0) = 0$ and its derivative $\kappa'>0$. %\qed 
\end{definition}  

\medskip

\noindent \red{ The KL property  holds for a large class of functions including semi-algebraic functions (e.g. real polynomial functions), vector or matrix (semi)norms (e.g., $\|\cdot\|_p$ with $p \geq 0$ rational number), logarithm functions,  exponential functions and  uniformly convex functions,  see \cite{BolDan:07} for a comprehensive list. In particular, if $\psi$ is semi-algebraic,  there  exist $q >1$ and $\sigma_q>0$ such that  $ \kappa$ in Definition \ref{def:kl} is of the form  $\kappa (t) = \sigma_q^{\frac{1}{q}} \frac{q}{q-1} t^{\frac{q-1}{q}}$  \cite{BolDan:07}. Then,  the KL property establishes the following local geometry of the nonconvex function $\psi$ around $\mathcal{X}$:
\begin{equation}
	\label{eq:kl}
	\psi(x) - \psi_*  \leq \sigma_q \|\nabla \psi(x)\|^q \quad   \forall x\!: \;  \text{dist}(x, \mathcal{X}) \leq \gamma, \; \psi_* < \psi(x) < \psi_* + \epsilon.  
\end{equation}
Moreover, the functions  $\psi(x) = - \ln(1-\|x\|^r)$ and $\psi(x) = \tan (\|x\|^r)$, with $r\geq 1$, satisfy KL with $q=\dfrac{r}{r-1}$ \cite{Wan:20}; any twice differentiable function $\psi:\mathbb{R}^{n} \to \mathbb{R}$ having the Hessian invertible at some stationary  point satisfies  KL  around that point with $q=2$ \cite{Wan:20}.} Note that the relevant aspect of the KL property is when $\mathcal{X}$ is a subset of \red{\textit{stationary points}  for $\psi$, i.e.,  $\mathcal{X} \subseteq \{x:\;  \nabla \psi (x)=0 \}$}, since it is easy to establish the KL property when $\mathcal{X}$ is not related to stationary points.  Another class of functions that  is important in this work is defined~next.

%Finally, the following  mean value inequality will be used in the sequel. 
%\begin{lemma}
%	\label{lemma:MVI}
%	Let $G:  \mathbb{R}^{n} \to \mathbb{R}^{m}$ be a continuously differentiable function and $J:  \mathbb{R}^{n} \to \mathbb{R}^{m\times n}$ be its Jacobian. Consider $U\sim \mathcal{U}$  and $x, x+Ud \in  \mathbb{R}^{n}$, with $d \in \mathbb{R}^{p}$. Then,  exists $y  \in  [x,x+Ud]$ such that:
%	\begin{equation*}
%		\|G(x+Ud) - G(x)\| \leq \|J(y)U\| \|d\|.
%	\end{equation*}
%\end{lemma}
%\proof{Proof}
%	See Appendix for a proof. \Halmos
%\endproof
%\begin{remark} 
%		\label{rem:MVI}
%		Suppose $V$ is a convex bounded set and $\mathcal{U}$ is bounded. Since G is continuously differentiable, we have that exists $M>0$ such that $\|J(y)U\| \leq M$, for all $y \in V$ and $U \sim \mathcal{U}$.
%\end{remark}}

\medskip

\begin{definition}
	\label{definition:3}
A function $\psi:\mathbb{R}^{n}\to \mathbb{R}$ is \textit{convex along a given family of subspaces $\mathbb{S}$}  if for any fixed $x \in \rset^n$ and  any matrix $U \in \mathcal{U}$ generating a subspace from the family  $\mathbb{S}$ the function restricted on this subspace:	
	\begin{equation}
		\phi^{x}_{U}(d) = \psi(x + Ud) 
		\label{eq:phi}
	\end{equation}
is convex on its corresponding  domain.
\end{definition}

\medskip

\noindent  Note that if $\psi$ is twice differentiable, then it is convex along a given family of subspaces if $U^{T}\nabla^2 \psi(x) U$ is positive semidefinite matrix for any $x$ and $U$ in the family of matrices $\mathcal{U}$ generating $\mathbb{S}$. One can easily notice that there are nonconvex functions $\psi$ which are convex along a given family of subspaces.  Similarly, we can define weak convexity  for a function $\psi$ along a given family of subspaces $\mathbb{S}$. 
%\red{ 
 %   Consider a differentiable function $g:\mathbb{R}^{n} \to \mathbb{R}$. We say $y^{*} \in \mathbb{R}^{n}$ is a stationary point when 
%    $   \nabla g (y^{*}) = 0.$}

%%%%%%%%%%%%%%%%%

\section{A stochastic coordinate proximal gradient algorithm}
\noindent  For minimizing the composite optimization problem \eqref{eq:prob},   where $f$ and $\psi$ are assumed  nonseparable  functions (both possibly  nonconvex), we propose    a pure stochastic  coordinate descent algorithm that requires sketching  the gradient   $\nabla f(x)$ on some stochastic subspace and computing the  prox of $\psi$ also along this subspace.  Hence, our Stochastic Coordinate Proximal Gradient (SCPG) algorithm is as follows:

\medskip

\begin{center}
	\noindent\fbox{%
		\parbox{12cm}{%
			\textbf{Algorithm 1 (SCPG)}:\\
			Give $\mathcal{U} \subset \mathbb{R}^{n \times p}$ a family  of matrices and starting point $x_{0} \in \rset^n$. \\
			For $k \geq 0$ do: \\
			1. \; Randomly sample $U_{k} \sim \mathcal{U}$.  Choose $\eta_{U_{k}} >0$ and update: \\
			\begin{equation*}
				\Hf = \left\lbrace\begin{array}{ll} \dfrac{L_{U_{k}} + \eta_{U_{k}}}{2} \; 
					\text{ if } \psi \; \text{convex along  subspaces  from } \mathcal{U}, \\
					L_{U_{k}} + \eta_{U_{k}} \; \text{ otherwise.}   \label{eq:76}
				\end{array}\right.   
			\end{equation*}
			2. \;  Solve the following subproblem:
			\begin{align}
				\label{eq:subproblem}
				\;\; d_{k} \in  \arg \min_{d \in \mathbb{R}^{p}} f(x_{k}) +  \langle U^{T}_{k} \nabla f(x_{k}), d \rangle + \dfrac{\Hf}{2} \|d\|^2  + \psi(x_{k} + U_{k}d).	
			\end{align}
			3. \; 	Update $x_{k+1} = x_{k} + U_{k}d_{k}$.
	}}
\end{center}

\medskip

\noindent   Note that, when the matrices $U_k$ are chosen as the scaled sampling matrices, SCPG coincides with the random block coordinate descent method, while for $U_{k} = I_{n}$, SCPG recovers the \textit{full} proximal gradient method, see Algorithm (46) in \cite{Nes:19}. Moreover, for $f \equiv 0$  and $\Hf=\eta_{U_{k}} >0$ our algorithm yields  the popular \textit{alternating minimization} (Gauss-Seidel ) method with regularization \cite{Ber:99,GriSci:00}, but  with general sampling strategies:
\begin{align} 
	\label{eq:am}	
	d_k  \in  \arg \min_d \psi(x_k + U_k d) + \dfrac{\Hf}{2} \|d\|^2,  \;\;  x_{k+1} = x_{k} + U_{k}d_{k}.  
\end{align}	   
However,  \cite{Nes:19} derives rates only in the convex settings. Moreover,   there are very few results ensuring that the iterates of a Gauss-Seidel method converges to a global minimizer, even for strictly convex functions, e.g., \cite{GriSci:00} presents only asymptotic convergence results. In this paper we derive convergence rates for the general algorithm SCPG in both convex and nonconvex settings, with and without $f$.  An important fact concerning our approach is that the convergence of  SCPG  works for any  $\eta_{U_{k}} $ greater than a fixed positive parameter which can be chosen arbitrarily small. In particular, according to step 1 in  SCPG  we can choose a larger stepsize when $\psi$ is convex along subspaces (see Definition \ref{definition:3}). \\

\noindent Moreover, our algorithm  requires  knowledge of the coordinatewise Lipschitz constants of the smooth term $f$ and  computation of the  proximal operator of $\psi$ only   along some subspace (i.e.,  the prox of the  \textit{partial} function $ \phi_{U_{k}}^{x_{k}}$ from \eqref{eq:phi}   at the sketching gradient~$U^{T}_{k} \nabla f(x_{k})$):
\begin{align}  
	& \arg \min_{d \in \mathbb{R}^{p}}  \phi_{U_{k}}^{x_{k}}(d)  + \frac{\Hf}{2} \|d +  \Hf^{-1}  \; U^{T}_{k} \nabla f(x_{k}) \|^2   = \text{prox}_{ \Hf^{-1}  \phi_{U_{k}}^{x_{k}}}  \left( - \Hf^{-1} \;  U^{T}_{k} \nabla f(x_{k}) \right).  	  
	\label{prox3.2}
\end{align} 

\noindent  Regardless the properties of the two functions $f$ and $\psi$, the subproblem  \eqref{eq:subproblem} in SCPG is convex provided that $\psi$ is (weakly) convex along subspaces\footnote{When $\psi$ is weakly convex along subspaces the subproblem is convex provided that $\Hf$ is chosen appropriately.} and then the  prox operator 	\eqref{prox3.2}   is well-defined (and unique) in this case.  For general nonconvex $\psi$, the prox 	\eqref{prox3.2}  has to be interpreted as a point-to-set mapping (recall that under   Assumption \ref{ass2}  and $\inf_{d \in \mathbb{R}^p}  \phi_{U_{k}}^{x_{k}}(d) > -\infty$,  the set $ \text{prox}_{ \Hf^{-1}  \phi_{k}}(v) $  is nonempty and compact for any $v$).\\

\noindent Note that  the proximal mapping is available in closed form for many useful functions, e.g., for  norm power $p\geq 2$ (see Section 6 for more examples).  Note that the prox of $\psi$ restricted to some subspace  (as required in SCPG) is less expensive computationally than the full prox of $\psi$.
	More precisely,  if $\psi$ is differentiable,  then solving the subproblem e.g., in the full proximal gradient method (algorithm (46) in \cite{Nes:19}), is equivalent to finding a full vector $d_{k} \in \mathbb{R}^{n}$ satisfying the system of $n$ nonlinear equations:
	\begin{equation}
		\nabla f(x_{k}) + \nabla \psi(x_{k} + d_k) +H_f d_{k}  = 0. \label{systEq}
		\vspace*{-0.1cm}
	\end{equation}
	
	\noindent  On other hand,   at each iteration of  our algorithm SCPG,  solving the subproblem \eqref{eq:subproblem} is equivalent to finding a vector $d_{k} \in \mathbb{R}^{p}$ satisfying the system of $p \ll n$ nonlinear equations:
	\vspace*{-0.1cm}
	\begin{equation}
		U_{k}^{T}(\nabla f(x_{k}) + \nabla \psi(x_{k} + U_{k}d_k)) +\Hf d_{k}  = 0.   \label{oneEq}
		\vspace*{-0.1cm}
	\end{equation}
	\noindent  Clearly, the computational cost of solving the nonlinear system  \eqref{oneEq}  is less than the computational cost of  solving the full nonlinear system  \eqref{systEq}.  Throughout the paper, we also assume boundedness of the level set: 
$$\mathcal{L}_{F}(x_{0})=\left\{x:  \,F(x)\leq F(x_{0}) \right\}. $$ 

\noindent The main challenge when analyzing convergence of SCPG  consists in connecting $d_k$ with the full gradient $\nabla F(x_k)$.  If this is possible, then one can use the cost as a Lyapunov function for the iterates of SCPG, see Lemma 	\ref{lem1} below.   Hence, our key concept in the analysis of algorithm SCPG is related to the quality of our subspace selection, i.e., we consider  some stochastic embedding of the gradient. This  stochastic embedding of the gradient  allows us to establish the connection between  $d_k$ and $\nabla F(x_k)$ in Lemma 	\ref{lem2}.  {Hence, next definition introduces the notion of well-aligned matrix.}   

\medskip

\begin{definition}
	The matrix $U_{k} \in \mathbb{R}^{n \times p}$ is \textit{well-aligned} if 
	\begin{equation}
		\|U^{T}_{k} \nabla F(x_{k}) \| \geq \alpha \| \nabla F(x_{k}) \|, \label{eq:1}
	\end{equation}
	for some $\alpha \in (0,1)$ independent of the iteration $k$.
\end{definition}

\medskip

\noindent The notation of well-aligned iterations has been also used in  \cite{CarRob:21}  for derivative-free methods in nonlinear least-squares optimization. In our convergence analysis we use the set of well-aligned updates in $K$ iterations, with $K$ fixed, denoted $\mathbb{A}_{K}$, i.e.: 
\[   \mathbb{A}_{K} = \{k \in [K]:  \; U_{k} \; \text{is well-aligned} \}.   \] 

\noindent In the sequel we also consider  the following assumption related to well-alignment (subspace quality): 
\begin{assumption}
	\label{ass1}
	At each iteration $k$  our subspace selection $U_{k}$ is well-aligned for some fixed $\alpha \in (0,1)$ with probability at least $1-\delta$, with $\delta \in (0,1)$, independently of  the filtration $\mathcal{F}_{k}:=\{U_{0},\cdots,U_{k-1}\}$.
\end{assumption}

\medskip

\noindent  {Fortunately, the random matrices theory \cite{Mah:16} provides us  many examples of  classes of  matrices that are probabilistic well-aligned.  Below we discuss some classes of such matrices:  e.g., scaled  sampling  matrices, Johnson-Lindenstrauss matrices and others (see also  \cite{CarRob:21,Mah:16,Shao:22}).}

\subsection{Scaled Sampling matrices}
First, we consider the class of scaled sampling matrices $\mathcal{U}$, which have one  non-zero entry per row in any random column.  They are defined as follows.
	
\medskip
	 
\begin{definition}	Matrix 	$U\in \mathbb{R}^{n \times p}$ is a scaled sampling matrix if, independently for each $i \in [p]$, we sample  $j \in [n]$  uniformly at random and let $U_{ij} = \sqrt{\dfrac{n}{p}}$ and the rest of the entries are zero. 
\end{definition}

\medskip

\noindent  The scaling factor is included so that  $\mathbb{E} [ \|U^Tx\| ]  = \|x\|$. Scaled sampling matrices are computationally inexpensive to apply to vectors so that embeddings based on them can be computed efficiently. We have the following probabilistic result for this class, see e.g.,   Lemma 4.4.13 in \cite{Shao:22}. 
	%\begin{lemma}
	Let $U \in \mathbb{R}^{n \times p}$ be a scaled sampling matrix. Let $\nu$ be the maximum nonuniformity of the gradient:
	\begin{equation*}
		\nu = \max_{x    } \left\{\dfrac{\|\nabla F(x)\|_{\infty}}{\|\nabla F(x)\|}  \right\}.
	\end{equation*}	
	Then, for any $\epsilon \in (0,1)$ and $\delta = e^{-\frac{\epsilon^2 p}{2n\nu^2}}$, we have: 
	\begin{equation*}
		\mathbb{P} [  \|U^T \nabla F(x)\|^{2}  \geq (1-\epsilon) \|  \nabla F(x) \|^2 ]  \geq 1 - \delta.
	\end{equation*}
	%\end{lemma}
	
	\noindent Note that  Assumption  \ref{ass1} holds if:  
	\begin{align*}
		&1-\epsilon \geq \alpha^2 \;  \Leftrightarrow \; 1 - \sqrt{\dfrac{2n\nu^2}{p} \log\dfrac{1}{\delta}} \geq \alpha^2 \; \Leftrightarrow \;  p \geq \dfrac{2n\nu^2}{\left(1-\alpha^2\right)^2} \log\dfrac{1}{\delta}. 
	\end{align*}
	\noindent Hence, $p$ depends on $\nu$. Note that we always have $1/n \leq \nu^2 \leq 1$. For well-behaved objective functions we  have $\nu^2  \simeq  1/n$ and then, in this case  $p$ is independent of the dimension  of the problem $n$. Otherwise, $p$ may depend on $n$.

%%%%%%%%%%%%%%%%%%%%%%%%

\subsection{Johnson-Lindenstrauss embeddings}
We can also consider  for $\mathcal{U}$ the class of Johnson-Lindenstrauss embeddings, i.e.   $U \in \mathcal{U}$ is  a Johnson-Lindenstrauss transform (JLT). For these matrices we can improve the requirement on $p$. Recall that  a random matrix $U \in \mathbb{R}^{n \times p}$ is an $(1 - \alpha,\delta)$ JLT if for any vector $v  \in \mathbb{R}^{n}$, we have:  
\begin{equation*}
	\mathbb{P} \left[\alpha \|v\|^{2} \leq \|U^T v \|^{2}\leq (2 - \alpha)\| v \|^{2}\right] \geq 1 - \delta. 
\end{equation*}
Some common examples are:

\medskip 

\noindent (JLT-1):  If $U$ is a random Gaussian matrix with independent entries  normally distributed, $U_{ij} \sim \mathcal{N}(0,1/p)$,  and $p = \mathcal{O} \left( (1- \alpha)^{-2}|\log(\delta)| \right)$, then $U$ is an $(1 - \alpha,\delta)$ JLT (see Theorem 2.13 in \cite{BouLug:12}).

\medskip 

\noindent (JLT-2): We say that $U$ is an $s$ hashing matrix if it has exactly $s$ nonzero entries per row (indices sampled independently), which take values $\pm 1/\sqrt{s}$ selected independently with probability $1/2$. Then, any $s$-hashing matrix $U$ with $s=\mathcal{O}  \left( (1 - \alpha)^{-1}|\log(\delta)| \right)$ and $p = \mathcal{O} \left((1- \alpha)^{-2}|\log(\delta)| \right)$ is  $(1 - \alpha, \delta)$ JLT \cite{KanNel:14}.\\

\begin{comment}
	{\color{red}
		\noindent (JLT-3): We can consider $U^{T} = PHD$, with $P\in\mathbb{R}^{p\times n}$ a sparse projection matrix where every entry is $0$ with probability $1-q$ and $P_{ij} \sim \mathcal{N}(0,1/q)$ , $D\in\mathbb{R}^{n\times n}$ is a diagonal matrix having $D_{ii} = \{\pm 1\}$ with equal probability and  $H\in\mathbb{R}^{n\times n}$ is the Walsh-Hadamard matrix.\\
		\noindent (JLT-4): Fix intergers $p$ and $n = 2^{r}$ with $p<n$ and $r=1,2,3,\cdots$.  A subsampled randomized Hadamard transform (SRHT) matrix is a $p \times n$ matrix of the form:
		\begin{equation*}
			U^{T} = \sqrt{\dfrac{n}{p}} RHD,
		\end{equation*}
		\noindent with $D\in\mathbb{R}^{n\times n}$  diagonal matrix having $D_{ii} = \{\pm 1\}$ with equal probability, $H\in\mathbb{R}^{n\times n}$ is a normalized Walsh-Hadamard matrix and $R \in \mathbb{R}^{p\times n}$ is a subset of rows from the $n \times n$ identity matrix, where the rows are chosen uniformly random and without replacement.  }
\end{comment}

\noindent Since  $\alpha^{2} \leq \alpha$ for all $\alpha \in (0,1)$, we have for any $U$ which is  $(1 - \alpha, \delta)$ JLT:
\begin{equation*}
	\mathbb{P} \left[\|U^T v \| \geq \alpha \|v \|\right] \geq 1 - \delta. 
\end{equation*}
\noindent  Then, for $v = \nabla F(x)$,  Assumption  \ref{ass1} holds if   $U$ is any $(1-\alpha,\delta)$ JLT.  Moreover, in the previous   examples  the dimension of the subspace $p$  satisfies: 
\begin{equation*}
	p \geq (1 -\alpha)^{-2}|\log(\delta)|. 
\end{equation*}

 \noindent Inequality above does not give a simple criterion for choosin $p$ in terms of $\alpha$ and $\delta$. In Figure 1 in \cite{KozBecDoo:21} we can see numerical evidence that $p$ can be choosen independent of $n$. 

%Hence, for JLT matrices the value of $p$ is independent of the  dimension $n$ (see also \cite{CarRob:21}).  

%%%%%%%%%%%%%%%%%%%%%%%%%
%%%%%%%%%%%%%%%%%%%%%%%%%

\section{Convergence analysis: nonconvex case}

\noindent In this section we derive convergence  rates for algorithm SCPG when the smooth function $f$ is nonconvex and nonseparable,  and $\psi$ is simple, but nonseparable, nonconvex and twice differentiable.   Let us first  recall a basic result related to the optimality conditions for a composite optimization problem. 
\begin{lemma}  \cite{Nes:19}
	Given the general composite optimization problem:
	\begin{align*}
		\min_{x \in \text{dom} \, \phi} \theta(x) + \phi(x),
	\end{align*}
	where $\theta$ is differentiable and $\phi$ is convex on the convex domain $\text{dom} \, \phi$. Then, $y^*$ is a stationary point if and only if the following optimality condition holds:
	\begin{align}
		\label{optcond}
		\langle \nabla \theta(y^*), y - y^* \rangle + \phi(y) - \phi(y^*) \geq 0 \quad \forall y \in \text{dom} \, \phi.
	\end{align}
\end{lemma}

\medskip 

\noindent Next lemma shows an important result, i.e., our algorithm is a descent method. %\begin{equation}
%\bar{H}_{U_{k}} = \left\lbrace\begin{array}{ll} 2H_{U_{k}} - L_{U_{k}} 
	%\text{ if } \psi \text{ is convex along the subspaces generated from } \mathcal{U}, \\
	%\text{ if } \phi_{k} \text{ defined in } \eqref{eq:69} \text{ is convex,}\\ 
	%H_{U_{k}} - L_{U_{k}} \text{ otherwise.} \label{eq:76}
	%\end{array}\right.   
	%\end{equation}
	%} 

\begin{lemma}
	\label{lem1}
	If  Assumption \ref{ass2} holds, then  the iterates of algorithm SCPG satisfies the following  descent inequality:
	\begin{align}
		\label{descent}
		F(x_{k+1}) \leq F(x_{k}) - \dfrac{\eta_{U_{k}}}{2} \|d_{k}\|^2  \quad \forall k \geq 0. 
	\end{align}
\end{lemma}

\proof{Proof} See Appendix for proof.   \Halmos
\endproof 

\medskip 

\noindent From previous lemma we see that one can choose a larger stepsize when $\psi$ is convex along the subspaces generated from $\mathcal{U}$, since $\Hf$ must  satisfy in this case  $\Hf> L_{U_{k}}/2$.  Otherwise, we should choose $\Hf$ satisfying $\Hf> L_{U_{k}}$. Let us now introduce some notations
that will be used in the sequel:  
\begin{equation}
	\bar{\nabla}^2\Psi(z_{1},\cdots,z_{n}) = \begin{bmatrix}
		\nabla_{1}^2\psi(z_{1})  \\
		\vdots \\		
		\nabla_{n}^2\psi(z_{n}) \label{hess}
	\end{bmatrix}, 
\end{equation}	

\noindent with $\nabla_{i}^2\psi(z_{i})$ being the $i$-th row of the hessian of $\psi$ at the point $z_{i} \in \mathbb{R}^{n}$, and

\begin{equation}
	\HF= \max_{z_{1},\cdots,z_{n} \in {\text{conv}}(\mathcal{L}_{F}(x_{0}))} \|U_{k}^{T}\bar{\nabla}^2\Psi(z_{1},\cdots,z_{n})U_{k} + \Hf I_{p\times p} \|_F,  \label{Hpsi}
\end{equation}
\noindent which is finite provided that the sublevel set $\mathcal{L}_{f}(x_{0})$ is bounded.  \red{ Note that  SCPG  does not require the knowledge of  the coordinatewise Lipschitz constants of the whole function  $F$ (or of the term  $\psi$) over $\mathcal{L}_{f}(x_{0})$. The constant $\HF$   only appears in the convergence rates. In some applications  the second term, $\psi$, although differentiable might have expensive  gradient evaluation or  the corresponding  coordinatewise Lipschitz constants  might be difficult to estimate; on the other hand,  if the computation of the prox for the second term $\psi$,  along a subspace, is easy, then  algorithm SCPG can be used (see Section 6 for such proximal friendly functions).} Using the well-aligned assumption, we  prove next  that the gradients of the objective function $F$ evaluated along the well-aligned iterations are bounded by the coordinate directions $\|d_k\|$. 

\begin{lemma} 
	\label{lem2}
	If  Assumptions \ref{ass2} and \ref{ass1} hold and $\mathcal{L}_{f}(x_{0})$ is bounded, then  the iterates of algorithm SCPG satisfy the following relation:
	\begin{align}
		\label{eq:23}
		\|\nabla F(x_{k})\|^2  \leq \frac{\HF^2}{\alpha^2}\|d_{k}\|^2  \quad \forall k \in \mathbb{A}_{K}.  
	\end{align}
\end{lemma}

\proof{Proof} See Appendix for proof.   \Halmos
\endproof
	
	\medskip 
\noindent For simplicity of the exposition, in Algorithm 1 we consider solving subproblem (9) globally, although our convergence results  (see the proofs of Lemmas 4 and 5) hold for a more relaxed condition which requires just finding a stationary  point $d_k$ for the subproblem (9) such that the objective function in subproblem (9) is smaller when evaluated in $d_k$ than in $0$. Note that such $d_k$ is usually easier to compute than finding the global minimum of (9). 
%	\noindent Now, using a  probabilistic  argument  from  \cite{ChuLu:02}  (see also \cite{CarRob:21,GraRoy:15} for a similar analysis used however for other types of optimization algorithms), we have that with high-probability the number of well-aligned iterations is bounded  below. 

\medskip 

	\noindent Now, using  probability arguments as in  \cite{CarRob:21,ChuLu:02,GraRoy:15}, but adapted to our new algorithm SCPG, we have that with high-probability the number of well-aligned iterations is bounded  below. 
	
	\begin{lemma}
		\label{lem3}
		If  Assumptions \ref{ass2} and \ref{ass1} hold, then with high-probability the following  bound on the number of well-aligned iterations holds: 
		\begin{equation}
			\label{eq:7}
			\mathbb{P} \left[ |\mathbb{A}_{K}| \geq (1- \beta)(1 - \delta)(K+1) \right] \geq 1 - e^{\left(-\frac{\beta^2}{2}(1-\delta)(K+1)\right)} \quad \forall \beta \in (0,1).	
		\end{equation}
	\end{lemma}
	
	\proof{Proof}
	See Appendix  for a proof. \Halmos
	\endproof

	%%%%%%%%%%%%%%%%%%%%%%%%%
	%%%%%%%%%%%%%%%%%%%%%5

	\subsection{Sublinear convergence}
	\noindent Before proving sublinear rate of convergence for  algorithm  SCPG in the nonconvex case,  we need to define the following constants (recall that we assume $\mathcal{U}$  bounded,  at least with high probability):
	\[  \HF \leq \HFmax   \;\; \quad  \forall k \geq 0. \]  
	Furthermore, consider: 
	\begin{equation}
		\eta_{\min} = \min_{k \in \mathbb{N} } \eta_{U_{k}} \;\; 
		\text{ and} \;\;   C = \dfrac{\alpha^2 \eta_{\min}}{2\HFmax^2}. \label{eq:40}
	\end{equation}
	
	\noindent Now, we are ready to prove a high-probability  sublinear bound for the minimum norm of subgradients generated by algorithm SCPG.  
		
	\begin{theorem} 
		\label{theo1}
		If  Assumptions \ref{ass2} and \ref{ass1} hold and $\mathcal{L}_{f}(x_{0})$ is bounded, then for the SCPG sequence $(x_k)_{k\geq 0}$  we have with high-probability the following convergence rate 	for all $\beta \in (0,1)$:  
		\begin{equation}
			\mathbb{P}\left[\min_{0\leq i\leq K} \|\nabla F(x_{i})\|^{2} \leq \dfrac{F(x_{0})-F^{*}}{C(1- \beta)(1 - \delta)(K+1)} \right]  \geq 1 - e^{-\frac{\beta^2}{2}(1-\delta) \cdot (K+1)}.   \label{eq:9}
		\end{equation}
	\end{theorem}
	
	\proof{Proof}
		From  Lemma \ref{lem1} and \eqref{eq:40}, we have
		\begin{eqnarray}
			\dfrac{\eta_{\min}}{2} \sum_{i=0}^{K}\|d_{i}\|^2 &\leq&  
			\sum_{i=0}^{K} \dfrac{\eta_{U_{i}}}{2} \|d_{i}\|^2  
			\leq \sum_{i=0}^{K} \left( F(x_{i}) - F(x_{i+1}) \right) \nonumber  \\ 
			&=&  F(x_{0}) - F(x_{K+1}) \leq  F(x_{0}) - F^{*}. \label{eq:46}
		\end{eqnarray}
		
		\noindent If we assume that the set of indexes corresponding to well-aligned iterations are given by $\mathbb{A}_{K}=\{k_{0},...,k_{J}\}$. Then, from Lemma \ref{lem2} it follows that: 
		\begin{align*}
			|\mathbb{A}_{K}| C \min_{0\leq i\leq J} \|\nabla F(x_{k_{i}})\|^2 &\leq  \sum_{i=0}^{J} \dfrac{\alpha^2 \eta_{\min}}{2\HF^2 } \| \nabla F(x_{k_{i}})\|^2  \\ 
			& \leq  \sum_{i=0}^{J} \dfrac{\eta_{\min}}{2} \|d_{k_{i}}\|^2 \leq  \dfrac{\eta_{\min}}{2} \sum_{i=0}^{K}\|d_{i}\|^2 \leq F(x_{0}) - F^{*}.
		\end{align*}
		
		\noindent Therefore, we get: 
		\begin{equation}
			|\mathbb{A}_{K}| \min_{0\leq i\leq K} \|\nabla F(x_{i})\|^2 \leq  \dfrac{1}{C}\left(F(x_{0})-F^{*}\right). \label{eq:10}
		\end{equation}
		
		\noindent Finally, from Lemma \ref{lem3}  we get:  
		\begin{equation}
			\mathbb{P}\left[  \min_{0\leq i\leq K} \|\nabla F(x_{i})\|^2 \leq \dfrac{|\mathbb{A}_{K}| \min_{0\leq i\leq k} \|\nabla F(x_{i})\|^2}{(1- \beta)(1 - \delta)(K+1)} \right] \geq 1 - e^{-\frac{\beta^2}{2}(1-\delta)(K+1)}.
			 \label{eq:11} 
		\end{equation}
		\noindent Hence, from \eqref{eq:10} and \eqref{eq:11} the statement  \eqref{eq:9} follows.
	\Halmos 
\endproof
	
	\medskip 
	
	\noindent Alternatively, the convergence rate from previous theorem can be stated in terms of a given accuracy  $\epsilon>0$ and probability $\gamma$, i.e., for 
	\begin{equation*}
		K \geq \max \left\lbrace \dfrac{F(x^{0}) - F^*}{\epsilon^2 \cdot C (1-\delta) (1-\beta)} - 1, \dfrac{2}{\beta^2(1-\delta)} \ln\dfrac{1}{\gamma} - 1\right\rbrace  
	\end{equation*}
	we have
	\begin{equation*}
		\mathbb{P}\left[\min_{0\leq i\leq K} \| \nabla F(x_{i}) \| \leq \epsilon \right]  \geq 1 - \gamma. 
	\end{equation*} 
	
	\noindent  This  convergence rate, where the probability $\gamma$ enters logaritmically in the estimate,   is similar to the ones derived in the literature for random coordinate descent \cite{Nes:10,RicTak:11}.    It is important to note that the high-probability sublinear bound from the previous theorem holds for objective functions having both terms $f$ and $\psi$ nonconvex and nonseparable.   
	
	\medskip 
	
	\begin{remark}
		\label{rem1}
	Usually, in the coordinate descent literature, where  $U$'s are submatrices of the identity matrix $I_n$ and $\psi$ is separable,   the complexity estimates are expressed in terms of the number of blocks (i.e.,  the ratio between  dimensions of the original problem and of the subproblems).  In our notations, this ratio can be defined as $N=n/p$. Moreover, the existing works for separable $\psi$ have derived convergence rates that depend linearly on such $N$ \cite{BecTet:13,FerRic:15,LuXia:14,Nes:10,NecCli:16,NecTak:20,RicTak:11}.    Let us also analyze the dependence on $N$ of the convergence rate \eqref{eq:9} of Theorem \ref{theo1} for  algorithm SCPG in the optimality given by  the norm of the gradient.   First, note that we can consider $1 - \delta \simeq 1$ and chose $\beta$ such that $\max \left(\dfrac{1}{\beta^2}, \dfrac{1}{1-\beta}\right) $ is small, e.g., $\beta = \dfrac{\sqrt{5} -1}{2}$ (the golden number). 	Now, let us analyze the rate in the norm of the gradient:
	
	\medskip 
	
	\noindent \textit{Scaled sampling matrices.}  Choosing $\alpha^2 = 1 - \nu \sqrt{2N \log(1/\delta)}$,  we get that with high-probability  the  convergence rate for  the norm of the gradient  is of order: 
	\begin{align*}
		\mathcal{O} \left( \frac{ \sqrt{\nu} N^{\frac{1}{4}}  \HFmax }{\sqrt{\eta_{\min} k}}\right).
	\end{align*} 
	\textit{JLT matrices.}  Choosing  $\alpha = 1/2$, then  $p\geq4|\log(\delta)|$ and we get that with high-probability  the  convergence rate for  the norm of the gradient  is of order:
	\begin{align*}
		O\left(\frac{ \HFmax}{\sqrt{\eta_{\min}k}}\right).
	\end{align*}
	In conclusion, when $U_k$ are scaled sampling matrices our convergence estimate depends on  $N$, while there is no explicit dependendence on $N$ for JLT matrices (although, in some particular cases, one can still have some hiden dependence on $N$ via the constant $\HFmax$).  	\Halmos 
	\end{remark}
	
	%For example, if the matrix $U_{k}$ is  Gaussian as in example JLT-1, we will have the dependence on $p$ and $n$ in the constant $\HFmax$. In this case, we can have
	%\begin{equation*}
	%\|U\| \asymp 1 + \sqrt{\dfrac{n}{p}}
	%\end{equation*}
	%with high probability.
	 %Assume $\psi$ is convex and the maximum eigenvalue of $\nabla^2\psi(x)$ is equal to 2, for all $x \in \mathcal{L}_{f}(x_{0})$. Moreover, consider $f(x) = \dfrac{1}{2}\|x\|^2$. Then, we have $\HFmax \leq \|U\|^2$ and  
	%\begin{equation*}
	%\HFmax \asymp 1 + \dfrac{n}{p}.
	%\end{equation*}
	%So, in this case, it is not an advantage to use the Johnson-Lindenstrauss matrices.	

	%%%%%%%%%%%%%%%%%%%%%%%%%%%5

	\subsection{Better  convergence under KL}
	In this section we prove convergence rates for our algorithm when $F$ satisfy additionally the KL property (in particular,  $F$ is uniformly convex).	Consider $(x_{k})_{k\geq 0}$ the sequence generated by algorithm SCPG and $(x_{k_{j}})_{j\geq 0}$ the subsequence of $(x_{k})_{k\geq 0}$ such that the iteration $k_{j}$ is well-aligned. Let us denote the set of limit points of the sequence $(x_{k_{j}})_{j\geq 0}$ by $\mathcal{X}(x_{0})$.  Next lemma derives some basic properties for  $\mathcal{X}(x_{0})$. 
	
	\begin{lemma}
		\label{LemmaKL}
		If Assumption \ref{ass2} holds and  the sublevel set $\mathcal{L}_{f}(x_{0})$ is bounded,  then  $\mathcal{X}(x_{0})$ is  a compact set,  $F (\mathcal{X}(x_{0})) = F_*$,  $F(x_{k})  \to F_{*}$ a.s., and additionally $\nabla F(\mathcal{X}(x_{0}))=0$,    $\| \nabla F(x_{k})\|  \to 0$~a.s.
	\end{lemma}
	
\proof{Proof} See proof in Appendix.   	\Halmos 	
\endproof
	
	\medskip 
	
	\noindent Next lemma shows convergence rates for a sequence satisfying a certain recurrence. 

	\begin{lemma}
		\label{lemma:rec}
		Let $\zeta > -1$, $c>0$  and  $\{\Delta_{k}\}_{k\geq0}$ be a  decreasing sequence of positive numbers satisfying the following recurrence:  
		
		\begin{equation}
			\Delta_{k} - \Delta_{k+1} \geq c \Delta_{k}^{\zeta+1} \quad \forall k \geq 0. \label{eq:78}
		\end{equation}
		
		\noindent Then, we have:
		
		\begin{itemize}
			\item[](i) For $c=1$ and  $\zeta> 0$ the sequence  $\Delta_k \to 0$ with sublinear rate: 
			\begin{equation}
				\Delta_{k} \leq \dfrac{ \Delta_{0}}{\left( \zeta \Delta_{0}^{\zeta} \cdot k + 1\right)^\frac{1}{\zeta}} \leq \left( \dfrac{1}{\zeta k}\right) ^{\frac{1}{\zeta}}. \label{eq:61}
			\end{equation} 
			
			\item[](ii) For $c\in(0,1)$ and  $\zeta = 0$ the sequence  $\Delta_k \to 0$ with linear rate: 
			\begin{equation}
				\Delta_{k} \leq \left(1-c\right)^{k}\Delta_{0}.  
				\label{eq:62}
			\end{equation}	 
			
			\item[](iii) 	For  $c>0$ and $\zeta \in (-1, 0)$ the sequence  $\Delta_k \to 0$ with superlinear rate: 
			\begin{equation}
				\Delta_{k+1} \leq \left(\dfrac{1}{1 + c \Delta_{k+1}^{\zeta}}\right) \Delta_{k}.  
				\label{eq:63}
			\end{equation} 	  
		\end{itemize}
		
	\end{lemma}
	
	\proof{Proof}
	See proof in Appendix. \Halmos
	\endproof
	
	\medskip
	
	\noindent For simplicity of the exposition, let us define the following constants: 	
	\begin{equation}
		\gamma_{1} = C \sigma_{q}^{-\frac{2}{q}},
			\quad \ 
		\gamma_{2} =  1- C \sigma_{2}^{-1} ,
	 \quad C_1 = \dfrac{2-q}{q} \text{ and } C_{2}=(1-\beta)(1-\delta).\label{eq:48}
	\end{equation}
	
	\begin{comment}
	\begin{equation}
		\gamma_{1} = \left\lbrace\begin{array}{ll} C \sigma_{q}^{-\frac{2}{q}},
			\text{ if } F \text{ satisfy KL property}, \\
			Cq^{\frac{2}{q}} \sigma_{q}^{\frac{2(q-1)}{q}}, \text{ if } F \text{ is uniformly convex.}
		\end{array}\right.  \label{eq:73}
	\end{equation}
	\begin{equation}
		\gamma_{2} = \left\lbrace\begin{array}{ll} 1- C \sigma_{2}^{-1} ,
			\text{ if } F \text{ satisfy KL property for} \; q=2, \\
			1-2\sigma_{2}C, \text{ if } F \text{ is uniformly convex for} \; q=2.
		\end{array}\right.  \label{eq:74}
	\end{equation}
	and
	\begin{eqnarray}
		&&C_1 = \dfrac{2-q}{q} \text{ and } C_{2}=(1-\beta)(1-\delta).\label{eq:48}
	\end{eqnarray}
	\end{comment}
	
	\noindent Recall that $\mathbb{A}_{K}=\{k_{0}, \cdots ,k_{j}, \cdots, k_J \}$ denotes the set of well aligned iterations until iteration $K$. In the next theorem we assume that $F$ satisfies the KL condition \eqref{eq:kl} with constant value $F_{*}$ and constant $q \in(1,2]$  around the limit points of the sequence $(x_{k_j})_{j\geq 0}$,  denoted $X(x_{0})$.  \textit{To the best of our knowledge there are very few studies analyzing  the convergence rate of stochastic algorithms under the KL property, see e.g., \cite{MauFadAtt:22}. The main difficulty comes from the fact that  the KL condition holds only locally, while for a stochastic algorithm we can usually  prove  almost sure convergence   $F(x_{k_j}) \inas F_{*}$ (see Lemma 	\ref{LemmaKL}), which means that there exists some measurable set $\Omega$ such that $ \mathbb{P}[\Omega] =1$ and for any  $\epsilon, \gamma>0$ and $\omega \in \Omega$  there exists $k_{j_{\epsilon,\gamma}}(\omega)$ such that for any $k_j \geq k_{j_{\epsilon,\gamma}}(\omega)$ we have   $F(x_{k_j}(\omega)) - F_{*} \leq \sigma_q \| \nabla F(x_{k_j}(\omega))\|^q$. Note that we cannot infer from this that $F(x_{k_j}) - F_* \leq \sigma_q \| \nabla F(x_{k_j})\|^q$ a.s. for $k_j$ large enough as $k_{j_{\epsilon,\gamma}}(\omega)$ is a random variable which, in general,  cannot be  bounded uniformly on $\Omega$.   
%	Using similar  arguments as in  Theorem 4.5 in \cite{MauFadAtt:22}, 
However, invoking measure theoretic arguments to pass from almost sure convergence to almost uniform convergence thanks to
Egorov’s theorem \cite{Rud:87}, we can prove that  for any $\bar \delta >0$,
%for any $ \delta,  \epsilon,\gamma >0$ there exist a measurable set $\Omega_{\delta}$, such that $ \mathbb{P}[\Omega_{\delta}] =1 - \delta$,  and $ k_{j_{\delta,\epsilon,\gamma}} > 0$ such that for all $\omega \in \Omega_{\delta}$ and $k_{j} \geq k_{j_{\delta,\epsilon,\gamma}} $ we have $F(x_{k_j}(\omega)) - F_* \leq \sigma_q \| \nabla F(x_{k_j}(\omega))\|^q$. Hence,  
with probability at least  $1-  \bar \delta$ the sequence $(x_{k_j})_{j\geq 0}$  satisfies KL   for all $k_{j} \geq k_{j_{\bar \delta,\epsilon,\gamma}}$}.  

	\begin{theorem}
		\label{theo:KL1}
		Let  Assumptions \ref{ass2} and  \ref{ass1} hold. Additionally, assume that the sublevel set $\mathcal{L}_{f}(x_{0})$ is bounded and  $F$ satisfies the KL property \eqref{eq:kl} on $\mathcal{X}(x_{0})$. Then, for any $\bar \delta>0,$ there exists  $k_{j_{ \bar \delta,\epsilon,\gamma}}>0,$  such that for all $K$ satisfying 
		\begin{equation}
			\label{Kcondition}
			K > \max \left(\dfrac{1+ j_{\bar  \delta,\epsilon,\gamma}}{(1-\beta)(1-\delta)} - 1, k_{j_{\bar  \delta,\epsilon,\gamma}} \right)\!,
		\end{equation} 		
		\noindent  the sequence $(x_k)_{k\geq 0}$ generated by SCPG  satisfies the following statements with probability at least~$1-\bar \delta$:  
		
		%\begin{itemize}
		%\item[]
		\noindent  (i) If $q\in (1,2)$, then the following sublinear rate holds:	
				\begin{eqnarray}
					\label{eq:13}
				\mathbb{P}\left[F(x_{K}) - F_* \leq \dfrac{1}{ \left( \gamma_{1}C_{1}\left(C_{2}(K+1) - (j_{ \bar \delta,\epsilon,\gamma}+1) \right) \right)^{\frac{q}{2-q}}} \right] 
					\geq 1 - e^{-\frac{\beta^2}{2}(1-\delta)(K+1)}, 
				\end{eqnarray}
		
	%	\begin{eqnarray*}
	%		\mathbb{P}\left[F(x_{K}) - F_* \leq \dfrac{F(x_{k_{\bar{j}}}) - F_*}{ \left( C_{1}\left(C_{2}(K+1) - (j_{ \delta,\epsilon,\gamma}+1) \right)\gamma_{1}( F(x_{k_{\bar{j}}}) - F_*)^{\frac{2-q}{q}}  \right)^{\frac{q}{2-q}}} \right] \\ 
	%		\geq 1 - e^{-\frac{\beta^2}{2}(1-\delta)(K+1)}, 
	%	\end{eqnarray*}
		% or
		%	\begin{equation*}
			%	\mathbb{P}\left[F(x_{K}) - F_* \leq \left(\dfrac{1}{C_{1}\gamma_{1}\left(C_{2}(K+1) - 1\right) }\right)^{\frac{q}{2-q}} \right]  \!\geq\!  1 - e^{-\frac{\beta^2}{2}(1-\delta)(K+1)},  
			%	\end{equation*}		
		%\noindent where $k_{0}$ is the first well-aligned iteration.
		
		%\item[]
		\noindent  (ii) If $q=2$, then the following linear rate holds: 	
		\begin{eqnarray}
			\mathbb{P}\left[ \!F(x_{K}) - F_{*} \!\leq\!  \gamma_{2}^{[C_{2}(K+1) - (j_{\bar  \delta,\epsilon,\gamma}+1)]}(F(x_{0}) - F_{*})\right] \! %\nonumber \\
			\geq 1 - e^{-\frac{\beta^2}{2}(1-\delta)(K+1)}.  
			\label{eq:51}
		\end{eqnarray}
		
		\noindent  (iii)   If $q>2$, then with probability at least~$1-\bar \delta$ the following superlinear rate holds:
		\begin{equation}
			F(x_{k_{j+1}}) - F_{*} \leq \left(\dfrac{1}{1 + C \sigma_{q}^{-\frac{2}{q}} \left( F(x_{k_{j+1}}) - F_{*} \right)^{\frac{2}{q}-1}} \right)  F(x_{k_{j}}) - F_{*}  \quad \forall k_{j} \geq k_{j_{ \bar \delta,\epsilon,\gamma}}
			\label{eq:53}
		\end{equation}
		%\end{itemize}	
	\end{theorem}
	
	\proof{Proof}
Denote  the set of indexes corresponding to well-aligned iterations until $K$  by $\mathbb{A}_{K}=\{k_{0},...,k_{J}\}$.  
 We assume that $F$ satisfies the KL condition \eqref{eq:kl} with constant value $F_{*}$ and  $q \in(1,2]$  around the limit points of the sequence $(x_{k_j})_{j\geq 0}$,  denoted $X(x_{0})$. From Lemma \ref{LemmaKL} we have that $F(x_{k_{j}}) \inas F_{*}$ and $\|\nabla F(x_{k_{j}}) \| \inas 0$, i.e., there exists a set $\Omega$ such that  $ \mathbb{P}[\Omega]  = 1$  and $F(x_{k_j}(\omega)) \to F_*(\omega)$ and $\|\nabla F(x_{k_j}(\omega)) \| \to 0$  for all $\omega \in \Omega$. This implies that for any  $\epsilon, \gamma>0$ and $\omega \in \Omega$  there exists $k_{j_{\epsilon,\gamma}}(\omega)$ such that for any $k_j  \geq k_{j_{\epsilon,\gamma}}(\omega)$ we have the KL condition  $F(x_{k_j}(\omega)) - F_{*} \leq \sigma_q \| \nabla F(x_{k_j}(\omega))\|^q$.   Hence,  we cannot infer from this that $F(x_{k_j}) - F_* \leq \sigma_q \| \nabla F(x_{k_j})\|^q$ a.s. for $k_j$ large enough as $k_{j_{\epsilon,\gamma}}(\omega)$ is a random variable which, in general,  cannot be  bounded uniformly on $\Omega$. However, using the Egorov’s theorem (see Exercise 16 in Chapter 3 in \cite{Rud:87} or  Theorem 4.5 in \cite{MauFadAtt:22}), we have that for any  $\bar  \delta>0$ there exists a measurable set $\Omega_{\bar  \delta} \subset \Omega$ satisfying $ \mathbb{P}[\Omega_{\bar  \delta}] \geq 1 - \bar \delta$ such that  $F(x_{k_j})$ converges uniformly to $F_{*}$ and $\nabla F(x_{k_j})$  converges uniformly to $0$  on the set $\Omega_{\bar  \delta}$. Since $F$ satisfies the KL property,  given $\epsilon, \gamma, \bar \delta > 0$, there exists a $k_{j_{ \bar \delta,\epsilon,\gamma}} > 0$ and $\Omega_{\bar \delta} \subset \Omega$ with $P[\Omega_{\bar \delta}] \geq 1 - \bar \delta$ such that  $\text{dist}(x_{k_j}(\omega), X(x_{0})) \leq \gamma, \; F_* < F(x_{k_j}(\omega)) < F_* + \epsilon$ and  
			\begin{equation}
				\label{eq:139}
				F(x_{k_j}(\omega)) - F_*  \leq \sigma_q \|\nabla F(x_{k_j}(\omega))\|^q  \qquad \forall k_j \geq k_{j_{\bar  \delta,\epsilon,\gamma}} \text{ and } \omega \in \Omega_{\bar  \delta}.
			\end{equation}
	
			\noindent Equivalently, using the indicator function for $ \Omega_{\bar  \delta}$, denoted $	\mathbbm{1}_{\Omega_{\bar  \delta}}$, we have:
			\begin{equation}
				\label{eq:83}
				\mathbbm{1}_{\Omega_{ \bar \delta}}	\left(  F(x_{k_j}) - F_* \right)^{\frac{2}{q}} \leq  \mathbbm{1}_{\Omega_{ \bar \delta}} \sigma_q^{\frac{2}{q}} \|\nabla F(x_{k_j})\|^2 \quad \forall k_j \geq k_{j_{\bar \delta,\epsilon,\gamma}}. 
			\end{equation}

%		\begin{eqnarray*}
	%		\dfrac{\eta_{\min}}{2} \|d_{k_{j}}\|^2 \leq F(x_{k_{j}}) - F(x_{k_{j}+1}). 
%		\end{eqnarray*}
		\noindent From Lemmas \ref{lem1}, \ref{lem2} and relation \eqref{eq:40}, we have: 
		\begin{align*}
			C \|\nabla F(x_{k_{j}})\|^2 & \leq \dfrac{\alpha^2 \eta_{\min}}{2 \HF^2} \| \nabla F(x_{k_{j}})\|^2 \leq  \dfrac{\eta_{\min}}{2} \|d_{k_{j}}\|^2   \leq  F(x_{k_{j}}) - F(x_{k_{j}+1}). %\label{eq:18}
		\end{align*}	
	
	\noindent This implies that 
	\begin{equation*}
		\mathbbm{1}_{\Omega_{\bar \delta}} \|\nabla F(x_{k_{j}})\|^2 \leq \mathbbm{1}_{\Omega_{\bar  \delta}} \dfrac{1}{C}	\left(  F(x_{k_{j}}) - F(x_{k_{j}+1})\right). 
	\end{equation*}
		
		\noindent Combining the inequality above with the KL property \eqref{eq:83},  we get:
		\begin{equation}
			\mathbbm{1}_{\Omega_{ \bar \delta}} \left(F(x_{k_{j}}) - F_{*}\right)^{\frac{2}{q}}  \leq \mathbbm{1}_{\Omega_{ \bar \delta}}  \sigma_{q}^{\frac{2}{q}} \|\nabla F(x_{k_{j}})\|^{2} \leq \mathbbm{1}_{\Omega_{\bar  \delta}}  \left(\dfrac{\sigma_{q}^{\frac{2}{q}}}{C}\right)\left( F(x_{k_{j}}) - F(x_{k_{j}+1})\right).  \label{eq:70}
		\end{equation}
		
%		\noindent If the function $F$ is uniformly convex, we obtain:
%		\begin{eqnarray}
%			\left(F(x_{k_{j}}) - F_{*}\right)^{\frac{2}{q}}  &\leq& \left( \dfrac{1}{q \sigma_{q}^{(q-1)}}\right)^{\frac{2}{q}}  \|\nabla F(x_{k_{j}})\|^{2} \nonumber \\ 
%			&\leq& \dfrac{1}{C} \left( \dfrac{1}{q \sigma_{q}^{(q-1)}}\right)^{\frac{2}{q}} \left( F(x_{k_{j}}) - F(x_{k_{j}+1})\right).  \label{eq:71}
%		\end{eqnarray}
		
		\noindent Since $(F(x_{k}))_{k\geq0}$ is a decreasing sequence (see Lemma \ref{lem1}),  we have $F(x_{k_{j+1}}) \leq F(x_{k_{j}+1})$ and then from \eqref{eq:70}   we get  the following recurrence valid for all $k_j \geq k_{j_{\bar  \delta,\epsilon,\gamma}}$:
		\begin{equation}
		\mathbbm{1}_{\Omega_{\bar  \delta}} 	\left(F(x_{k_{j}}) - F_{*}\right) - \mathbbm{1}_{\Omega_{\bar  \delta}} \left(F(x_{k_{j+1}}) - F_{*}\right) \geq \mathbbm{1}_{\Omega_{\bar  \delta}}  C \sigma_{q}^{-\frac{2}{q}} \left(F(x_{k_{j}}) - F_{*}\right)^{\frac{2}{q}}.  \label{eq:64}
		\end{equation}
%		\noindent while by \eqref{eq:71}, for uniform convexity we get the recurrence:
%		\begin{equation}
%			\left(F(x_{k_{j}}) - F_{*}\right) - \left(F(x_{k_{j+1}}) - F_{*}\right) \geq  C \sigma_{q}^{\frac{2(q-1)}{q}} q^{\frac{2}{q}} \left(F(x_{k_{j}}) - F_{*}\right)^{\frac{2}{q}}.  \label{eq:72}
%		\end{equation}

		\noindent First,  let us consider that $q\in(1,2)$.
		Multiplying both sides of the inequality \eqref{eq:64}  by $\left(C \sigma_{q}^{-\frac{2}{q}}\right)^{\frac{q}{2-q}}$,  denoting  $\Delta_{k_{j}} = \mathbbm{1}_{\Omega_{\bar  \delta}} \gamma_{1}^{\frac{q}{2-q}}\left(F(x_{k_{j}}) - F_{*}\right)$ and  considering $\gamma_{1}$ defined in \eqref{eq:48}, we obtain:
		\begin{equation}
			\Delta_{k_{j}} - \Delta_{k_{j+1}} \geq  (\Delta_{k_{j}})^{\frac{2}{q}}   \quad \forall k_j \geq k_{j_{\bar \delta,\epsilon,\gamma}}.  \label{eq:52}
		\end{equation} 
	 Considering the second inequality in \eqref{eq:61} for $\zeta = \frac{2-q}{q}>0$,  we get:
			\begin{eqnarray*}
				& \Delta_{k_J} &\leq \dfrac{1}{\left( \zeta (J- j_{\bar \delta,\epsilon,\gamma})\right)^\frac{1}{\zeta}} 
				 \iff \mathbbm{1}_{\Omega_{\bar  \delta}} \gamma_{1}^{\frac{q}{2-q}}\left(F(x_{k_{J}}) - F_{*}\right) \leq \dfrac{ 1}{\left( \frac{(J- j_{\bar \delta,\epsilon,\gamma})(2-q)}{q} \right)^\frac{q}{2-q}} \\
				& \iff&  \mathbbm{1}_{\Omega_{ \bar \delta}}\left(  F(x_{k_{J}}) - F_{*}\right)  \leq \dfrac{ 1}{\left( \frac{(J- j_{\bar \delta,\epsilon,\gamma})(2-q)}{q} \gamma_{1}\right)^\frac{q}{2-q}}
			\end{eqnarray*}
			
		 %Using  again that $(F(x_{k}))_{k\geq0}$ is decreasing and 
		 \noindent Using $|\mathbb{A}_{K}| = J+1$, we obtain:
			\begin{align}
				\label{eq:12}
				 \mathbbm{1}_{\Omega_{\bar  \delta}}\left(  F(x_{k_{J}}) - F_{*}\right) \leq \dfrac{1}{ \left(\frac{\left( |\mathbb{A}_{K}|-1 - j_{\bar \delta,\epsilon,\gamma}\right)(2-q)}{q} \gamma_{1}\right)^{\frac{q}{2-q}}}.
			\end{align}
%			\noindent If we denote 
%			\begin{equation*}
%				D = \gamma_{1} ( F(x_{k_{\bar{j}} }) - F_*)^{\frac{2-q}{q}},
%			\end{equation*}
%			we get
%			\begin{equation*}
%				F(x_{K}) - F_* \leq \dfrac{F(x_{k_{\bar{j}}}) - F_*}{\left( 1 + \frac{(2-q) \left(|\mathbb{A}_{K}| - 1- \bar{j}\right)  }{q} D  \right)^{\frac{q}{2-q}}}.
%			\end{equation*}
			
			\noindent From Lemma \ref{lem3}, we have  for all $\beta \in (0,1)$:
			\begin{equation}
				\mathbb{P}[|\mathbb{A}_{K}|- 1 \geq (1- \beta)(1 - \delta)(K+1) - 1] \geq 1 - e^{-\frac{\beta^2}{2}(1-\delta)(K+1)}.  \label{eq:66}
			\end{equation}
			Hence, 	 if $K > \dfrac{1+j_{\bar \delta,\epsilon,\gamma}}{(1-\beta)(1-\delta)} - 1$, we get:
			\begin{equation}
				\mathbb{P}\left[ \dfrac{1}{|\mathbb{A}_{K}|- 1 - j_{\bar \delta,\epsilon,\gamma}} \leq \dfrac{1}{(1- \beta)(1 - \delta)(K+1) - 1 - j_{\bar \delta,\epsilon,\gamma} } \right]  \geq 1 - e^{-\frac{\beta^2}{2}(1-\delta)(K+1)}. \label{eq:47}
			\end{equation}
			
			\noindent From \eqref{eq:12} and the inequality above, we get 
			\begin{equation*}
				\mathbb{P}\left[\mathbbm{1}_{\Omega_{\bar  \delta}}\left(  F(x_{k_{J}}) - F_{*}\right) \leq \dfrac{1}{ \left( \gamma_{1}C_{1}\left(C_{2}(K+1) - (j_{\bar \delta,\epsilon,\gamma}+1) \right) \right)^{\frac{q}{2-q}}} \right]  \geq 1 - e^{-\frac{\beta^2}{2}(1-\delta)(K+1)}.
			\end{equation*}
		
			\noindent Using basic probabilistic arguments, we have that: 
			\begin{equation*}
				\mathbb{P}\left[ F(x_{k_{J}}) - F_{*} \leq \dfrac{1}{ \left( \gamma_{1}C_{1}\left(C_{2}(K+1) - (j_{\bar \delta,\epsilon,\gamma}+1) \right) \right)^{\frac{q}{2-q}}} \right]  \geq 1 - e^{-\frac{\beta^2}{2}(1-\delta)(K+1)}.
			\end{equation*}
			Since the sequence $(F(x_{k}))_{k\geq 0}$ is decreasing, the first statement of the theorem follows.  
			
			%On the other hand, if we use the second inequality in \eqref{eq:61}  for $\alpha = \frac{2-q}{q}>0$ and $q\in(1,2)$, we get:	 
			%\begin{align*}
			%& \Delta_{k_{J}} \leq \left( \dfrac{1}{\alpha J}\right)^{\frac{1}{\alpha}} 
			%\iff
			%\gamma_{1}^{\frac{q}{2-q}} (F(x_{k_J}) \!- F_*) \leq \! \left(\frac{q}{(2-q)J}\right)^{\frac{q}{2-q}} \\
			%&\iff F(x_{k_J}) - F_* \leq \left(\dfrac{ q}{(2-q)\gamma_{1} J} \right)^{\frac{q}{2-q}}.
			%\end{align*}	
			%\noindent Using again  that $(F(x_{k}))_{k\geq0}$ is decreasing and $|\mathbb{A}_{K}| = J+1$, we obtain: 
			%\begin{align*}
			%F(x_{K}) - F_*  \leq \left(\dfrac{q}{(2-q)\gamma_{1}(|\mathbb{A}_{K}|-1)}\right) ^{\frac{q}{2-q}}.
			%\end{align*}
			%\noindent Similarly, from \eqref{eq:47}, if $K > \dfrac{1}{(1-\beta)(1-\delta)} - 1$, we get: 
			%	\begin{equation*}
				%	\mathbb{P}\left[F(x_{K}) - F_* \leq \left(\dfrac{q}{(2-q)\gamma_{1}\left(C_{2}(K+1) - 1\right) }\right)^{\frac{q}{2-q}} \right]  \!\geq\!  1 - e^{-\frac{\beta^2}{2}(1-\delta)(K+1)},  
				%	\end{equation*}
			%	and thus the second statement of the theorem follows. 
			
			\medskip 
			
			\noindent Second, let us consider $q=2$. From  \eqref{eq:62} and \eqref{eq:64}, we have: 	
			\begin{align*}
				\mathbbm{1}_{\Omega_{\bar \delta}}\left( F(x_{k_{J}})- F_*\right)  \leq \mathbbm{1}_{\Omega_{ \bar \delta}}\left(1-\dfrac{C}{\sigma_{2}}\right)^{J-j_{\bar \delta,\epsilon,\gamma}} (F(x_{k_{j_{\bar \delta,\epsilon,\gamma}}}) - F_*). 
			\end{align*}
			
		%	\noindent Moreover, by \eqref{eq:62} and \eqref{eq:72}, we get
		%	\begin{align*}
		%		F(x_{k_{J}})- F(x^*) \leq \left(1 - 2C \sigma_{2} \right)^{J-\bar{j}} (F(x_{k_{0}}) - F(x^*)). 
		%	\end{align*}
			
			\noindent Using $|\mathbb{A}_{K}| = J+1$, we obtain: 
			\begin{align}
				\mathbbm{1}_{\Omega_{\bar \delta}}\left( F(x_{k_{J}})- F_*\right) \leq \mathbbm{1}_{\Omega_{ \bar \delta}} \gamma_2^{|\mathbb{A}_{K}|-(j_{\bar \delta,\epsilon,\gamma} +1)} (F(x_{ k_{\bar j_{\delta,\epsilon,\gamma}}  }) - F_*). \label{eq:50}
			\end{align}
			with $\gamma_2$ defined in \eqref{eq:48}.  If $K > \dfrac{1+j_{\bar \delta,\epsilon,\gamma}}{(1-\beta)(1-\delta)} - 1$ by $\eqref{eq:66}$, we get: 
			\begin{equation*}
				\mathbb{P}\left[ \! \gamma_2^{|\mathbb{A}_{K}|- (j_{\bar \delta,\epsilon,\gamma} + 1)}   \leq  \gamma_2^{(1- \beta)(1 - \delta)(K+1) - (j_{\bar \delta,\epsilon,\gamma} + 1)}\right] \! \geq 1 - e^{-\frac{\beta^2}{2}(1-\delta)(K+1)}. 
			\end{equation*}
			\noindent  Using this probability bound in   \eqref{eq:50},   
			we get
			\begin{equation*}
			\mathbb{P}\left[ \! \mathbbm{1}_{\Omega_{\bar  \delta}}\left( F(x_{k_{J}})- F_*\right) \leq \mathbbm{1}_{\Omega_{ \bar \delta}} \gamma_2^{(1- \beta)(1 - \delta)(K+1) - (j_{\bar \delta,\epsilon,\gamma} + 1)}  (F(x_{k_{j_{\bar \delta,\epsilon,\gamma}}}) - F_*)    \right] \! \geq 1 - e^{-\frac{\beta^2}{2}(1-\delta)(K+1)}. 
		\end{equation*}
			
		\noindent Using  that $F(x_k) \leq F(x_0)$ for all $k \geq 0$, we finally get:  	
		\begin{equation*}
			\mathbb{P}\left[ \!  F(x_{k_{J}})- F_* \leq  \gamma_2^{(1- \beta)(1 - \delta)(K+1) - (j_{\bar \delta,\epsilon,\gamma} + 1)}  (F(x_{0}) - F_*)    \right] \! \geq 1 - e^{-\frac{\beta^2}{2}(1-\delta)(K+1)}. 
		\end{equation*}	
			
			\noindent Finally, since  $(F(x_{k}))_{k\geq 0}$ is decreasing, the second statement of the theorem also follows.  
			
			\medskip 
			
			\noindent Third, if $q>2$, using $F(x_{k_{j+1}}) \leq F(x_{k_{j}})$ in \eqref{eq:52}, we obtain:
			\begin{equation*}
				\Delta_{k_{j}} - \Delta_{k_{j+1}} \geq  (\Delta_{k_{j+1}})^{\frac{2}{q}}.
			\end{equation*} 
			Using  $\zeta = \frac{2}{q}-1 \in (-1,0)$, then we get \eqref{eq:63},   which yields the third  statement of the theorem.  \Halmos
	\endproof

	\begin{comment}
{	\color{red}
	\noindent Let us analyze the convergence rate when the function $F$ satisfy KL property. Define $\text{cond}_q = \sigma_{q}^{\frac{2}{q} } \HFmax^2 \eta_{\min}^{-1}$.  Then, we notice that the dependence on $N$  (with $N=n/p$) of the convergence rates from Theorem \ref{theo:KL1}  for  algorithm SCPG in the optimality given by  the function values are of the following orders: \\
	
	\noindent \textit{Scaled sampling matrices.}  Choosing $\alpha^2 = 1 - \nu \sqrt{2N \log(1/\delta)}$,  we get the rate in probability is of order: 
	\begin{align*}
		\mathcal{O} \left( \left( \frac{\nu \sqrt{N}  \text{cond}_q }{k}\right)^{\frac{q}{2-q}} \right).
	\end{align*} 
	\textit{JLT matrices.}  Choosing  $\alpha = 1/2$, then  $p\geq4|\log(\delta)|$ and we get that the rate in probability is of order:
	\begin{align*}
		O\left( \left( \frac{ \text{cond}_q }{k}\right)^{\frac{2}{2-q}} \right).
	\end{align*}

	\noindent For $q=2$ and scaled sampling matrices, choosing $\alpha^2 = 1 - \nu \sqrt{2N \log(1/\delta)}$, in this case we have
		\begin{align*}
		\mathcal{O} \left( \left(  1  +  \dfrac{\nu \sqrt{N\log(1/\delta)  }-1}{\text{cond}_2} \right)^k  \right). ??
	\end{align*} 

	\noindent while for  \textit{Johnson-Lindenstrauss matrices} is of  order :  
	\begin{equation*}
		\mathcal{O}\left(\left(1 - \dfrac{1}{  \text{cond}_2}\right)^k\right).
	\end{equation*}
\end{comment}	
	
	\medskip 
	
	\begin{remark}
	Note that if the function $F$ satisfies the Polyak-Lojasiewicz (PL)  inequality or it is uniformly convex of degree $q > 1$, then  the KL inequality  \eqref{eq:kl}  holds for $F$ at the minimum value $F^{*}$ for all $x \in \mathbb{R}^{n}$ (see  \cite{BolDan:07} for the  PL and 	\eqref{eq:probConv} for  the uniform convex functions). Hence, in these settings we have a global relation of the form  \eqref{eq:kl} for $F$, not a local one as in the general KL property.  Under these settings (i.e., PL or uniform convex functions) the results of the previous Theorem 	\ref{theo:KL1} are stronger as they hold  for all $K > 1/((1-\beta)(1-\delta)) - 1$ and for $\bar \delta =0$.  Moreover, similar conclusions as in Remark \ref{rem1} can be derived regarding the dependence on $N$  for the convergence rates obtained in  Theorem 	\ref{theo:KL1}.  
	\end{remark}

	\section{Convergence analysis: convex case}
	In this section we assume  that  the composite objective function  $F=  f + \psi$ is convex. Note that we do not need to impose convexity on $f$ and $\psi$ separately. Denote the set of optimal solutions of \eqref{eq:prob} by $X^{*}$ and let  $x^{*}$ be an element of this set. Define also: 
	\begin{equation*}
		R := \max_y \min_{x^{*} \in X^{*}} \{\|y - x^{*}\|: F(y) \leq F(x_{0})\}, 
	\end{equation*} 
	which is a measure of the size of the sublevel set of $F$ in $x_{0}$.  	Recall that  $k_{0}$ denotes the first well-aligned iteration. In these settings  we get the following convergence result.  
	
	\begin{theorem}
		\label{the:conv}
	If Assumptions \ref{ass2} and  \ref{ass1} hold, the sublevel set $\mathcal{L}_{f}(x_{0})$ is bounded and additionally  $F$ is  convex, then for any number of iterations $K$ satisfying 
			\begin{equation*}
				K > \max\left\{ \dfrac{1}{(1-\beta)(1-\delta)} - 1, k_{0}\right\}
		\end{equation*}		
		\noindent the sequence $(x_k)_{k\geq 0}$ generated by algorithm SCPG  has with high-probability the following sublinear rate in function values:  
		{\small\begin{equation}
				\mathbb{P}\left[
				F(x_{K}) - F(x^{*}) \leq \dfrac{1}{C}\left(\dfrac{R^2}{(1- \beta)(1 - \delta)(K+1) - 1}\right)
				\right]  \geq 1 - e^{-\frac{\beta^2}{2}(1-\delta)(K+1)}.  \label{eq:17}
		\end{equation}}	
	\end{theorem}
	
	\proof{Proof}
From  Lemma \ref{lem2}, for $k\in\mathbb{A}_{K}$, we have that:  
		\begin{equation*}
			\|d_{k}\|^{2} \geq \dfrac{\alpha^2}{\HF^2} \|\nabla F({x_{k}})\|^{2}.
		\end{equation*}  
		
		\noindent Since  $\mathbb{A}_{K}=\{k_{0},...,k_{J}\}$. Combining the inequality above, Lemma \ref{lem1} and relation \eqref{eq:40},  we get:
		\begin{equation}
			F(x_{k_{j}}) - F(x_{k_{j}+1}) \geq \dfrac{\eta_{\min}}{2} \|d_{k_{j}}\|^2 \geq C \|\nabla F({x_{k_{j}}})\|^2 .  \label{eq:49}
		\end{equation}

		\noindent On the  other hand, since $F$ is convex, we have:
		\begin{eqnarray*}
			F(x^{*}) - F(x_{k_{j}}) \geq \langle \nabla F(x_{k_{j}}), x^{*} - x_{k_{j}+1} \rangle &\geq& - \| \nabla F(x_{k_{j}})\| \|x_{k_{j}+1} - x^{*} \| \\
			&\geq& - \| \nabla F(x_{k_{j}})\| R.
		\end{eqnarray*}
		
		\noindent Hence
		\begin{equation}
			\|\nabla F(x_{k_{j}})\| \geq \dfrac{F(x_{k_j}) - F(x^{*})}{R}. \label{eq:19}
		\end{equation}
		
		\noindent By \eqref{eq:49} and \eqref{eq:19}, we obtain: 
		\begin{equation*}
			\left(F(x_{k_{j}}) - F(x^{*})\right)  - \left(F(x_{k_{j}+1}) - F(x^{*})\right)  \geq  C \dfrac{\left( F(x_{k_j}) - F(x^{*})\right)^{2}}{R^{2}}.
		\end{equation*} 
		
		\noindent Since $(F(x_{k}))_{k\geq0}$ is a decreasing sequence (see Lemma \ref{lem1}),  we have $F(x_{k_{j+1}}) \leq F(x_{k_{j}+1})$ and thus:
		\begin{equation*}
			\left(F(x_{k_{j}}) - F(x^{*})\right)  - \left(F(x_{k_{j+1}}) - F(x^{*})\right)  \geq  C \dfrac{\left( F(x_{k_{j}}) - F(x^{*})\right)^{2}}{R^{2}}.
		\end{equation*}   
		
		\noindent Multiplying both sides by $C/R^{2}$, we further get: 
		\begin{equation*}
			\dfrac{C \left(F(x_{k_{j}}) - F(x^{*})\right)}{R^{2}}  - \dfrac{C \left(F(x_{k_{j+1}}) - F(x^{*})\right)}{R^{2}}  \geq  \left[\dfrac{C \left( F(x_{k_{j}}) - F(x^{*})\right)}{R^{2}}\right]^{2}.
		\end{equation*} 
		
		\noindent We denote: 
		\begin{equation*}
			\dfrac{C \left( F(x_{k_{j}}) - F(x^{*})\right)}{R^{2}} = \Delta_{k_{j}}.
		\end{equation*}
		
		\noindent Then, we obtain the following recurrence: 
		\begin{equation*}
			\Delta_{k_{j}} - \Delta_{k_{j+1}} \geq \left(\Delta_{k_{j}}\right)^{2}.
		\end{equation*}
		
		\noindent  From Lemma \ref{lemma:rec}, relation \eqref{eq:61},  we obtain $\Delta_{k_{J}} \leq 1/J$,     or equivalently:  
		\begin{align*}
			\dfrac{C\left( F(x_{k_{J}}) - F(x^{*})\right)}{R^{2}} \leq \dfrac{1}{|\mathbb{A}_{K}| - 1}  
			\iff  F(x_{k_{J}}) - F(x^{*}) \leq \dfrac{1}{C}\left(\dfrac{ R^2}{|\mathbb{A}_{K}| - 1}\right).
		\end{align*}
		
		\noindent Using again the fact that $(F(x_{k}))_{k\geq0}$ is decreasing, we get: 
		\begin{equation}
			F(x_{K}) - F(x^{*}) \leq \dfrac{1}{C}\left(\dfrac{ R^2}{|\mathbb{A}_{K}| - 1}\right). \label{eq:35}
		\end{equation}
		
		\noindent On the other hand, from Lemma \ref{lem3} we have: 
		\begin{equation*}
			\mathbb{P}\left[\dfrac{1}{|A_{K}|- 1} \leq \dfrac{1}{(1- \beta)(1 - \delta)(K+1) - 1}\right] \geq 1 - e^{-\frac{\beta^2}{2}(1-\delta)(K+1)}. 
		\end{equation*}
		
		\noindent for all $\beta \in (0,1)$. If $K > \dfrac{1}{(1-\beta)(1-\delta)} - 1$, then $(1- \beta)(1 - \delta)(K+1) - 1>0$ and from \eqref{eq:35}  our convergence rate  \eqref{eq:17} follows.  \Halmos
	\endproof
	
	\medskip 
		
	\noindent  We notice that the dependence on $N$, where $N=n/p$,  in the convergence rate in function values from Theorem \ref{the:conv}  for  algorithm SCPG  is for scaled sampling matrices  of order:
	\begin{align*}
		\mathcal{O} \left( \frac{\nu \sqrt{N}  \HFmax^2  }{\eta_{\min} k}\right),
	\end{align*} 
	while for Johnson-Lindenstrauss matrices of  order: 
	\begin{align*}
		O\left(\frac{ \HFmax^2}{ \eta_{\min} k}\right).
	\end{align*} 
	
	%\begin{align*}
	%	O\left(\frac{n \cdot  \HFmax^2 \eta_{\min}^{-1} }{k }\right), 
	%\end{align*}

%	\begin{align*}
%		O\left(\frac{\HFmax^2\eta_{\min}^{-1} }{k }\right).
%	\end{align*}
	
\noindent Hence, similar conclusions as in Remark \ref{rem1} can be also derived  in this case. Finally,  it is important to note that an immediate byproduct  of our previous convergence analysis (Theorems 1, 2 and 3)  leads to novel convergence results and rates for the popular alternating minimization (Gauss-Seidel) algorithm with a proper regularization (i.e., $f=0$) in the convex and nonconvex settings  with  general sampling strategies.

\section{Simulations}
\noindent	In the numerical  experiments we consider several applications:  the subproblem in the cubic Newton method \cite{NesPol:06}, the smallest eigenvalue of a matrix \cite{CarDuc:17} and the  logistic regression with quadratic or cubic regularization \cite{Mit:97}. In the sequel, we describe these problems  and present extensive numerical results. Note  that our composite problem \eqref{eq:prob} permits to handle general coupling functions $\psi(x)$,   with $x= (x_1,\cdots ,x_N)$, e.g.:  (i) $\psi(x) = \| A x\|^\ell$, with $\ell \geq 2$ and $A$ linear operator (in particular, $\psi(x_1,x_2) = \|A_1x_1 - A_2 x_2\|^\ell$, see \cite{LuFreNes:18}); (ii) when solving the  subproblem in higher order methods (including cubic Newton) recently popularized by Nesterov \cite{Nes:19} (where $\psi(x) = \|x\|^\ell$); (iii)  when  minimizing an objective function that is relatively smooth w.r.t. some (possibly unknown) function $h$, see \cite{LuFreNes:18}.  \red{In all the simulations, we choose  scaled sampling matrices $U_{k}$ and for a random algorithm we report the average time/full iterations after 5 runs.}

%%%%%%%%%%%%%%%%%%%%%%%%%%%
	
\subsection{Cubic Newton subproblem}
		\noindent  One of our main motivations for analyzing coordinate descent schemes for composite problems with nonseparable terms  came from the need of having fast algorithms for solving  the subproblem in the cubic Newton method \cite{NesPol:06}, which is supported by global efficiency estimates for general classes of optimization problems. Recall that in  the cubic Newton method at each iteration one needs to minimize an objective function of the form:
	\begin{align}
		\label{opt_prob}
		\min_{x \in \mathbb{R}^n} F(x)   :=  \frac{1}{2}   \langle Ax, x \rangle  + \langle b, x \rangle + \frac{M}{6} \|x \|^3,
	\end{align}
	where $A \in \mathbb{R}^{n\times n}, b\in\mathbb{R}^{n}$ and $M>0$ are given. As discussed in Section 2,  the function   $\psi(x) = \frac{M}{6} \|x \|^3$ is uniformly convex, but it is nonseparable. Moreover, in this case $f(x) =  \langle b, x \rangle + \frac{1}{2}   \langle Ax, x \rangle$ is smooth. Hence, this problem fits into our general model \eqref{eq:prob} and we can use algorithm SCPG to solve it. 
	%In this section   we discuss  some implementation details for the subproblem of SCPG when solving the quadratic minimization with cubic regularization  and  perform detailed  numerical simulations. 	

 \medskip 

\noindent 	\textbf{Implementations details:}\\
In the SCPG method, we need to solve the following subproblem at each iteration:
    \begin{align*}
        \min_{d\in \mathbb{R}^{n}}  \dfrac{M}{6}\|x + Ud\|^3 + \langle U^{T} \nabla f(x), d \rangle + \dfrac{H_{f,U}}{2}\|d\|^2.
    \end{align*}

    \noindent   Hence,  from the optimality condition of the subproblem above, we obtain: 
    \begin{equation}
    \label{eq:22}
    U^{T}\nabla f(x)  + \HfU d^{*} + \dfrac{M}{2} \|x+Ud^{*}\| U^{T}(x+Ud^{*}) = 0.  
    \end{equation}
    From the equality above, we have:
    \begin{equation}
    \label{eq:20}
        \left(\HfU I_{p} + \dfrac{M}{2} \|x+Ud^{*}\| U^{T}U \right) d^{*} = - \left(U^{T}\nabla f(x) + \dfrac{M}{2} \|x+Ud^{*}\| U^{T} x \right).
    \end{equation}

\noindent Note that, if we know  $\mu = \|x+Ud^{*}\|$, we can solve the linear system of equations \eqref{eq:20} very efficiently in order to obtain  $d^{*}$, even  in $\mathcal{O}(p)$ operations.  Indeed, if the matrix $U$ is chosen as scaled sampling matrix, then $U^TU$, via some permutation,  is a block diagonal matrix with blocks rank one matrices of the form $\mathbb{1}_{p_j} \mathbb{1}_{p_j}^T$, with $\mathbb{1}_{p_j}\in \mathbb{R}^{p_{j}}$ a vector of ones, such that $\sum_{j=1}^{\bar{p}} p_{j} = p$.  Then, solving \eqref{eq:20} is equivalent to finding solutions to  linear  systems of the form:
    \begin{align}
        \label{eq:21}
        \left(H_{f,U}I_{p_{j}} +\dfrac{M\mu n}{2p} \mathbb{1}_{p_j} \mathbb{1}_{p_j}^T\right) d^{(j)} = g^{(j)} \quad \forall j=1:\bar{p}. 
    \end{align}

  \noindent When $p_{j}=1$ the solution of \eqref{eq:21} is trivial to compute. When $p_{j}>1$ the linear system \eqref{eq:21} can be solved in $\mathcal{O}(p_j)$ operations. Indeed, \eqref{eq:21} is equivalently to:
    \begin{align*}
&H_{f,U}\mathbb{1}_{p_j}^{T}d^{(j)} +\dfrac{M \mu n}{2p} (\mathbb{1}_{p_j}^{T}\mathbb{1}_{p_j}) \mathbb{1}_{p_j}^{T}d^{(j)} = \mathbb{1}_{p_j}^{T}g^{(j)} \;\; \Leftrightarrow \;\; 
        \mathbb{1}_{p_j}^{T}d^{(j)} = \dfrac{\mathbb{1}_{p_j}^{T}g^{(j)}}{H_{f,U} + (\mathbb{1}_{p_j}^{T}\mathbb{1}_{p_j})M\mu n/2p}.
    \end{align*}
    Knowing now the value of $\mathbb{1}_{p_j}^{T}d^{(j)}$, we can compute easily $d^{(j)}$ as:
    \begin{align*}
        d^{(j)} = \dfrac{1}{H_{f,U}} \left( g^{(j)} - \dfrac{M\mu n}{2p}(\mathbb{1}_{p_j}^{T}d^{(j)}) \mathbb{1}_{p_j}\right). 
    \end{align*}
    
 \noindent  Note that the total cost of computing $d^{(j)}$ is $\mathcal{O}(p_j)$ and thus the total cost of computing the solution of the system \eqref{eq:20} is $\mathcal{O}(p)$. Further, let us see how we can compute $\mu$.  Note that if  $U$ is the scaled sampling matrix, then $UU^{T}$ is a diagonal matrix with $\bar{p} \leq p$ nonzero entries on the diagonal. Then, from   \eqref{eq:22} we have:
    \begin{equation*}
       x+Ud^{*} =  \left(\HfU I_{n} + \dfrac{M\mu}{2}UU^{T} \right)^{-1}\left( \HfU x - UU^{T}\nabla f(x)\right),
    \end{equation*}
and it is easy to compute the inverse of the above matrix, since it is just a diagonal matrix.  Let us denote $J$ the index set corresponding to the indices of the nonzeros entries in the diagonal of $UU^{T}$ and let $u_{j}$ be  the corresponding values,  with $j\in J$. Taking the norm square in the equality above and using the notation $\mu= \|x+Ud^{*}\|$, we obtain the following nonlinear equation in $\mu$:
    \begin{align*}
        \mu^{2} = \sum_{j \in J} \dfrac{\left(\HfU x^{(j)} - u^{(j)} \nabla_{j}f(x)\right)^2}{\left(\HfU + M\mu u^{(j)}/2 \right)^2} +\sum_{i=1,i\notin J}^{n} (x^{(i)})^2.
    \end{align*}
Then, $\mu$ is the positive root of the above scalar nonlinear equation. 
%$ h(\mu) = 0$, where $h$ is given by:
%    \begin{align*}
%       h(\mu) = \mu^{2} - \sum_{j\in J} \dfrac{\left(\HfU x^{(j)} - u^{(j)} \nabla_{j}f(x)\right)^2}{\left(\HfU + M\mu u^{(j)}/2 \right)^2} - \sum_{i=1,i\notin J}^{n} (x^{(i)})^2.
%    \end{align*}
Note that there are very efficient methods for finding the root of a scalar nonlinear equation, e.g., Newton method. Hence, for these settings, the overall complexity in the SCPG algorithm at each iteration is usually proportional to $p$. {Similarly,   when $\psi(x) = |a^{T}x|^\ell$, with $\ell>2$, and $U$ is chosen as scaled sampling matrix, we can also solve the corresponding subproblems efficiently in the SCPG algorithm combining the same reasoning as above with   the Sherman-Morrison formula.}
\color{black}

\medskip 

\noindent We compare algorithm SCPG with the methods proposed in \cite{CarDuc:17} and \cite{Nes:19} which were designed  for solving  subproblem \eqref{opt_prob}.   Note that, since  SCPG recovers  Algorithm 46 in \cite{Nes:19}, we can also apply that method  in the nonconvex case.   Recall that the update rule in \cite{CarDuc:17} is given by: 
	\begin{equation*}
		x_{k+1} = (I - \eta A - \dfrac{M}{2}\eta\|x_{k}\|I)x_{k} - \eta b,
	\end{equation*}
	
	\noindent with $0 < \eta \leq \dfrac{1}{4\|A\| + 2MR}$ and $R = \dfrac{\|A\|}{M} + \sqrt{\dfrac{\|A\|^2}{M^2} + \dfrac{2\|b\|}{M}}$. Moreover, the update rule in \cite{Nes:19} is:
	\begin{equation*}
		x_{k+1} = \dfrac{2}{2H + M \mu} (Hx_{k} -Ax_{k}-b),
	\end{equation*}
	
	\noindent with $\|A\| \leq H$ and $\mu \geq 0$ given by the solution of the equation $M/2\mu^2 + H \mu - \|Hx_{k} -Ax_{k}-b\| = 0$. In our implementations we used the parameters $\eta = \dfrac{1}{4\|A\| + 2MR}$ and $H = \|A\|$. Moreover, in SCPG we considered  $\Hf = \|U^{T}_{k}AU_{k}\|$. Following  \cite{CarDuc:17}, the starting point  is: 
	\begin{equation*}
		x_{0} = -r\dfrac{b}{\|b\|},  \text{ with } r = -\dfrac{b^{T}Ab}{M\|b\|^2} + \sqrt{\left( \dfrac{b^{T}Ab}{M\|b\|^2}\right)^2 + \dfrac{2\|b\|}{M}}.
	\end{equation*}
	
\noindent 	 In the  experiments  the vector $b\in\mathbb{R}^{n}$ is generated from a standard normal distribution $\mathcal{N}(0,1)$ and the matrix $A \in \mathbb{R}^{n \times n}$ was generated as $A = Q^{T}BQ$, where $Q \in \mathbb{R}^{n \times n}$ is an orthogonal matrix  and $B \in \mathbb{R}^{n \times n}$  is a diagonal matrix (with positive entries in the convex case and negative entries in the nonconvex case, respectively). Moreover, we consider different values for $n$ ranging from $10^4$ to $10^6$ and also several values for  $p$ in SCPG.  In Table 1, ``**" means that the corresponding algorithm took more than 7 hours.  Table 1  presents the number of full iterations (number of gradient evaluations $\nabla f$) and the CPU  time, in seconds, for each method in order to achieve:
\begin{equation*}
	\|\nabla F(x_{k})\| \leq 10^{-2}.
\end{equation*}  
From Table 1 we can see that SCPG has much better performance than the full gradient type  algorithms from \cite{CarDuc:17} and \cite{Nes:19} in both number of full iterations and CPU time.

	\begin{table}[h!]
		\centering    
		\begin{tabular}{| c| c| c | c | c | c | c | c | c |}
			\hline
			\multicolumn{9}{|c|}{ Convex} \\
			\hline
			\text{n} & \text{M} & \multicolumn{5}{|c|}{SCPG} & \cite{CarDuc:17} & \cite{Nes:19}    \\  \hline
			& & p=0.1\%n & p=0.5\%n & p=1\%n & p=2.5\%n & p=5\%n & & \\ \hline
			& 1 & 992 & 1254 & 860 & 1169 & 891 & 65564 & 8199   \\
			& & 336.3 & 192.1 & \textbf{109.3} & 168.4 & 126.5 & 1024.1 & 129.3 \\ \cline{2-9} 
			$10^4$ & 0.1 & 5280 & 5372 & 5349 & 3341 & 5904 & 251460 & 31437  \\
			& & 1648.8 & 776.3 & 655.6 & \textbf{470} & 906.7 & 3892.1 & 488  \\ \cline{2-9} 
			& 0.01 & 10660  & 6531 & 8990 & 6778 & 7027 & 918651 & 114836 \\ 
			& & 5030.8 & 1115 & 1002.7 & 822.2 & \textbf{818.2} & 14595 & 1897.9 \\ \hline
			& 1 & 323 & 604 & 607 & 572 & 679 & 40959 & 5124 \\ 	
			& & \textbf{504.2} & 819.8 & 848.6 & 864.3 & 1227.3 & 11483 & 776.5 \\ \cline{2-9}
			$10^5$ & 0.1 & 1387 & 1577 & 1380 & 1388 & 1416 &  & 18508 \\
			& & 2634.4 & 2035 & \textbf{1779.6} & 1955.1 & 2366 & ** & 2568 \\ \cline{2-9}
			& 0.01 & 5453 & 4417 & 5071 & 5035 & 5580 &  & 66966 \\
			& & 6032.8 & \textbf{4366.1} & 5064.4 & 5886.5 & 8025 & ** & 11148 \\ \hline
			& 1 & 1333  & 963 & 1415 &  1293 & 1220 &  & 11896  \\ 
			& & 20275 & 6676 & 8215 & 6214 & \textbf{5424} & ** & 7532 \\ \cline{2-9}
			$10^6$ & 0.1 & 2285 & 2760 & 3154 & 3353 & 3683 & &  \\ 
			&  & 18671  & 18067 &  17235 & \textbf{15014} & 15290 & ** & ** \\ \cline{2-9}
			& 0.01 &  &  & 6881 & 7103 & 7423 &  & 149880 \\ 
			& & ** & ** & 21471 & 17018 & \textbf{12725} & ** & 24986  \\ \hline 
			\multicolumn{9}{|c|}{ Nonconvex} \\
			\hline
			\text{n} & \text{M} & \multicolumn{5}{|c|}{SCPG} & \cite{CarDuc:17} & \cite{Nes:19}    \\  \hline
			& & p=0.1\%n & p=0.5\%n & p=1\%n & p=2.5\%n & p=5\%n & & \\ \hline
			& 1 & 1124 & 1161 & 1069 & 1269 & 1152 & 108892 & 13619   \\
			& & 297.9 & 113.6 & \textbf{79.2} & 119.1 & 107.9 & 1809.5 & 117.8  \\ \cline{2-9} 
			$10^4$ & 0.1 & 3893 & 3583 & 2970 & 3044 & 4166 &  & 530571  \\
			& & 824.5 & 298.7 & \textbf{205.2} & 275.6 & 382.8 & ** & 4562.4  \\ \cline{2-9} 
			& 0.01 & 4385  & 4527 & 4430 & 3916 & 4659 & & \\ 
			&  & 1087.2 & 419.3 & \textbf{315.9} & 345.1 & 405.6 & ** & **  \\ \hline
			& 1 & 773 & 852 & 825 & 816 & 852 & 54725 & 6847 \\ 	
			& & 834.6 & \textbf{762.6} & 821.2 & 945.6 & 1079.8 & 16494 & 1081.5 \\ \cline{2-9}
			$10^5$ & 0.1 & 765 & 2248 & 1640 & 2219 & 1890 &  &  \\
			& & \textbf{907.9} & 2036 & 1577.4 & 2437.7 & 2366.3 & ** & ** \\ \cline{2-9}
			& 0.01 & 3853 & 3792 & 3738 & 3742 & 3862 &  &  \\
			& & 4629.3 & 3695.4 & \textbf{3194.2} & 3606.6 & 4180.6 & ** & ** \\ \hline
			& 1  & 1962 & 1576 & 1451 & 1516 & 1705 &  & 35013 \\
			&    & 20583 & 6039.4 & 4577.3 & 4290.9 & \textbf{4028.8} & ** & 4134.2  \\ \cline{2-9}
			$10^6$ & 0.1 &  4981 & 5290 & 5416 & 5610 & 5957 &  &  \\
			&  & 18095 & 17280 & 15884 & 14385 & \textbf{11688} & ** & **  \\ \cline{2-9}
			& 0.01 &  &  & 6134 & 7004 & 9809 &  &  \\ 
			&  & ** & ** & 19827 & 18501 & \textbf{15138} & ** & ** \\ \hline
		\end{tabular}
		\begin{center}
			\caption{Number of full iterations and CPU time (in seconds) for algorithms SCPG, \cite{CarDuc:17} and \cite{Nes:19} on cubic Newton.}
		\end{center}
		\label{table:4}
	\end{table}

	\begin{comment}
		\begin{table}[h!]
			\centering    
			\begin{tabular}{| c | c | c| c| c | c | c | c |}
				\hline
				Group & Name &\text{n} & \text{p}	& \text{M} & \text{SCPG} &  \cite{Nes:19} & Power Method (crit=$10^{-7}$)  \\  \hline
				\multirow{1}{*}{ GHS\underline{ }indef } & \multirow{1}{*}{ a0nsdsil }  & 80016 	& 283 &   1   &  571  &  803 & 1063   \\
				& & & & & 576.3 & 4.4 & 7.9 \\ 
				& & & & & -42.08 & -42.08 & -42.08 \\  \hline
				\multirow{1}{*}{ TSOPF } & \multirow{1}{*}{ FS\underline{ }b9\underline{ }c1 }  & 2454 	& 50 &   1   &  143  &  174 & 125   \\
				& & & & & 4.8 & 0.03 & 0.04 \\ 
				& & & & & -5306.2 & -5306.2 & -5306.2 \\ \hline	
				\multirow{1}{*}{ AG-Monien } & \multirow{1}{*}{whitaker3}  & 9800 	& 99 &   1   &  1921  &  4069 & 251895   \\
				& & & & & 297.9 & 16.6 & 410.4 \\ 
				& & & & & -2.98 & -2.98 & -2.98 \\ \hline			
			\end{tabular}
			\begin{center}
				\caption{Number of full iterations and time for algorithms SCPG, \cite{CarDuc:17} and \cite{Nes:19}: Eigenvalue Problem.}
			\end{center}
		\end{table}	
	\end{comment}

%%%%%%%%%%%%%%%%%%%%%%%

\subsection{Smallest eigenvalue of a matrix}
\noindent In the second set of experiments we want to find the smallest eigenvalue of an indefinite matrix. As proved in \cite{CarDuc:17}, if a matrix $A$ has at least one negative eigenvalue, we can use formulation  \eqref{opt_prob} with $b = 0$ to find the smallest eigenvalue $\lambda_{\min}$. We compare our algorithm with the power method. We consider matrices  $A$ from University of Florida Sparse Matrix Collection \cite{DavHu}, more specific, we consider matrices from the Schenk\underline{ }IBMNA Group. The starting point $x_{0}$ was generated from a standard normal distribution $\mathcal{N}(0,1)$. Moreover, we consider different values for $p$ in the SCPG method. Figures \ref{fig:1} and \ref{fig:2} show the error  $ 	|\lambda_{k} - \lambda_{\min}|$ for the minimum eigenvalue along time, in seconds. As we can see from these figures, SCPG algorithm has much  better performance (i.e., finds an highly  accurate approximation of $\lambda_{\min}$ in much less time)  than power method for the matrices taken from Schenk\underline{ }IBMNA~group.

	\begin{comment}
		\begin{figure}[!ht]
			\centering
			\includegraphics[width=1\textwidth,height=5cm]{c-30_Schenk_IBMNA} 
			\caption{Group: Schenk\underline{ }IBMNA; Matrix: c-30, n=5321 }
		\end{figure}
		\begin{figure}[!ht]
			\centering
			\includegraphics[width=1\textwidth,height=5cm]{c-50_Schenk_IBMNA} 
			\caption{Group: Schenk\underline{ }IBMNA; Matrix: c-50, n=22401 }
		\end{figure}
		\begin{figure}[!ht]
			\centering
			\includegraphics[width=1\textwidth,height=5cm]{c-54_Schenk_IBMNA} 
			\caption{Group: Schenk\underline{ }IBMNA; Matrix: c-54, n=31793 }
		\end{figure}
		\begin{figure}[!ht]
			\centering
			\includegraphics[width=1\textwidth,height=5cm]{c-61_Schenk_IBMNA} 
			\caption{Group: Schenk\underline{ }IBMNA; Matrix: c-61, n=43618 }
		\end{figure}
		\begin{figure}[!ht]
			\centering
			\includegraphics[width=5.9cm,height=4cm]{c-65_Schenk_IBMNA} 
			\caption{Group: Schenk\underline{ }IBMNA; Matrix: c-65, n=48066 }
		\end{figure}
	\end{comment}

	\begin{figure}[!ht]
		\centering
		\includegraphics[width=8cm, height=6.5cm]{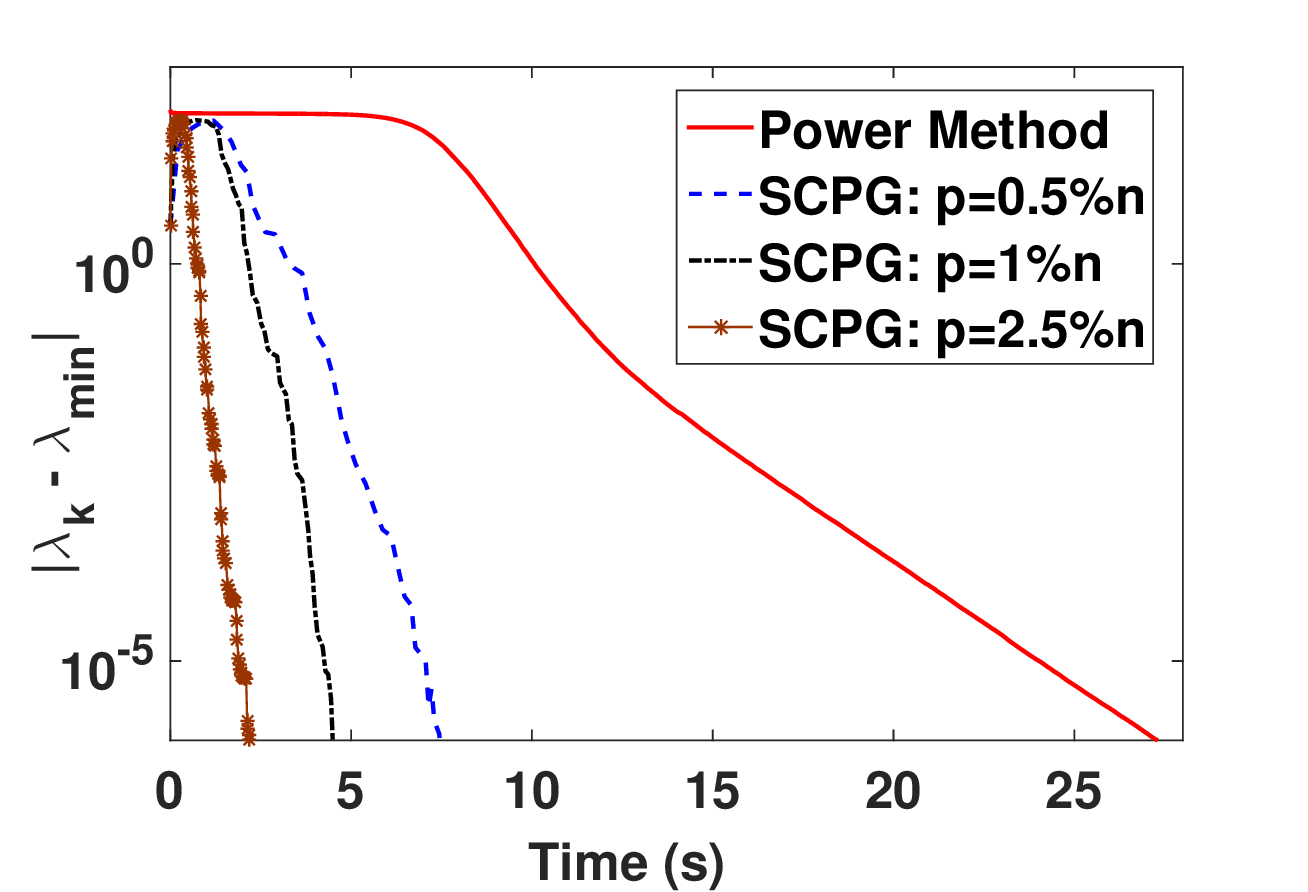}
		\hspace{-0.7cm}
		\includegraphics[width=8cm, height=6.5cm]{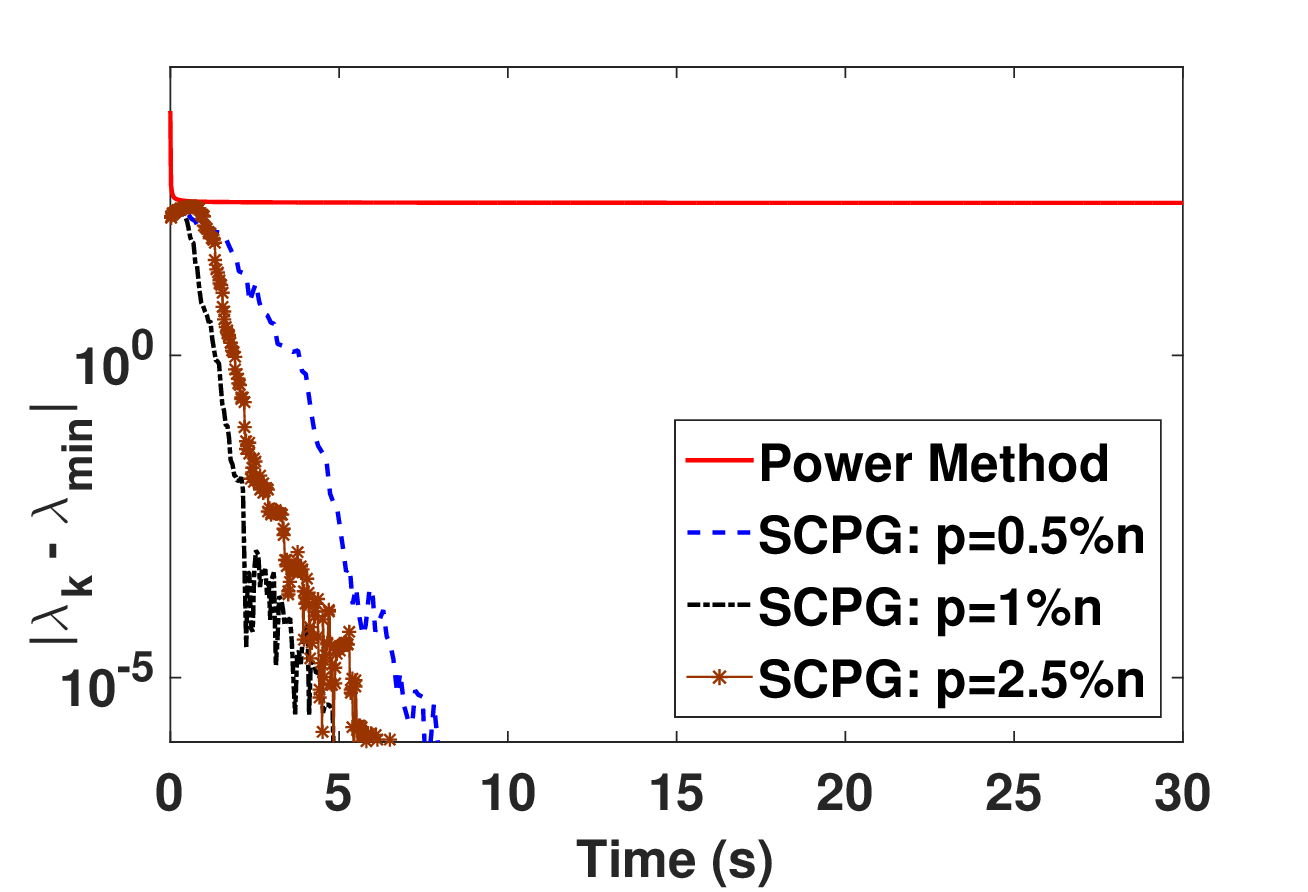} 
		\caption{Comparisson between SCPG and power method for the smallest eigenvalue  on group Schenk\underline{ }IBMNA.  Matrices: c-50 (left) n=22401  and c-54 (right) n=31793. }
		\label{fig:1}
	\end{figure}
%\vspace{-1.1cm}	
	\begin{figure}[!ht]
		\centering
		\includegraphics[width=8cm, height=6.5cm]{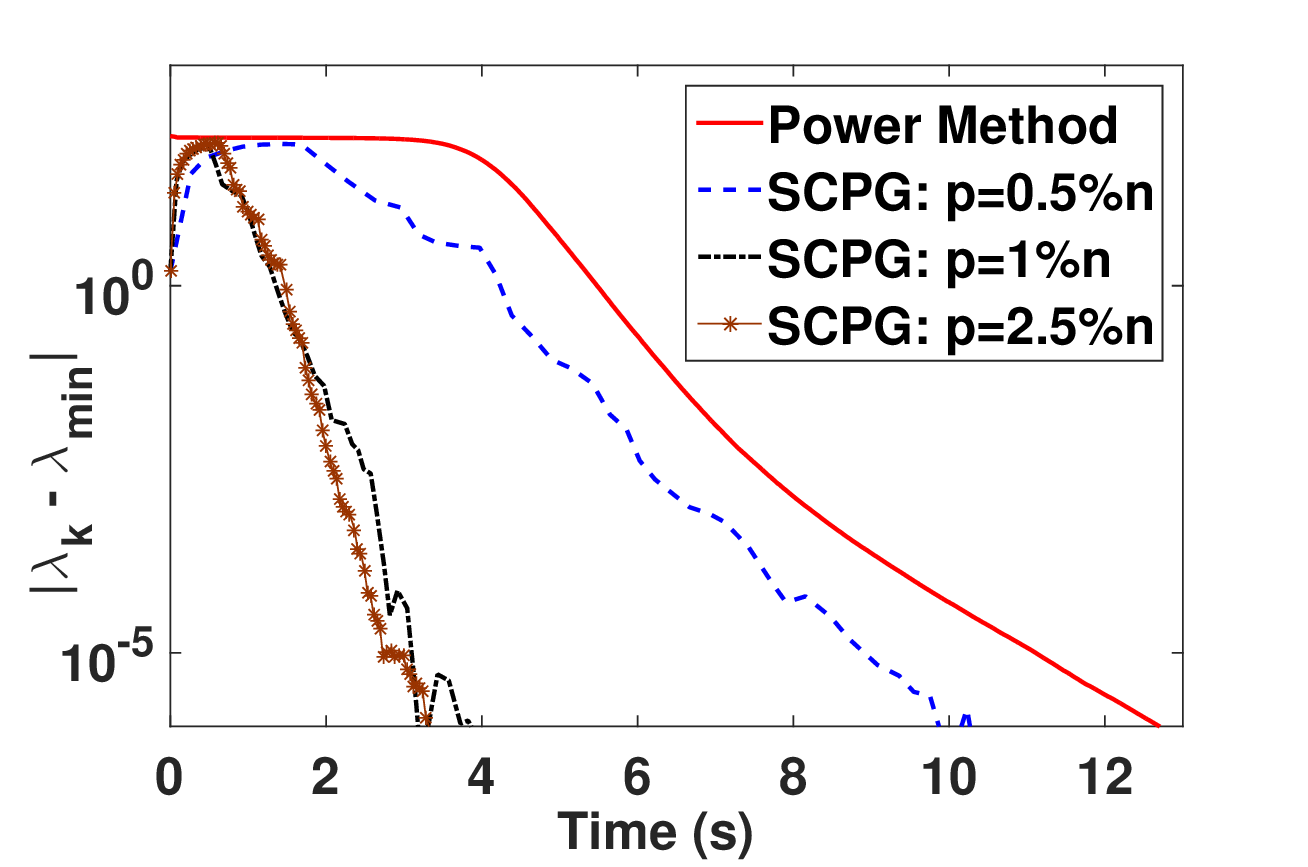}
		\hspace{-0.7cm}
		\includegraphics[width=8cm, height=6.5cm]{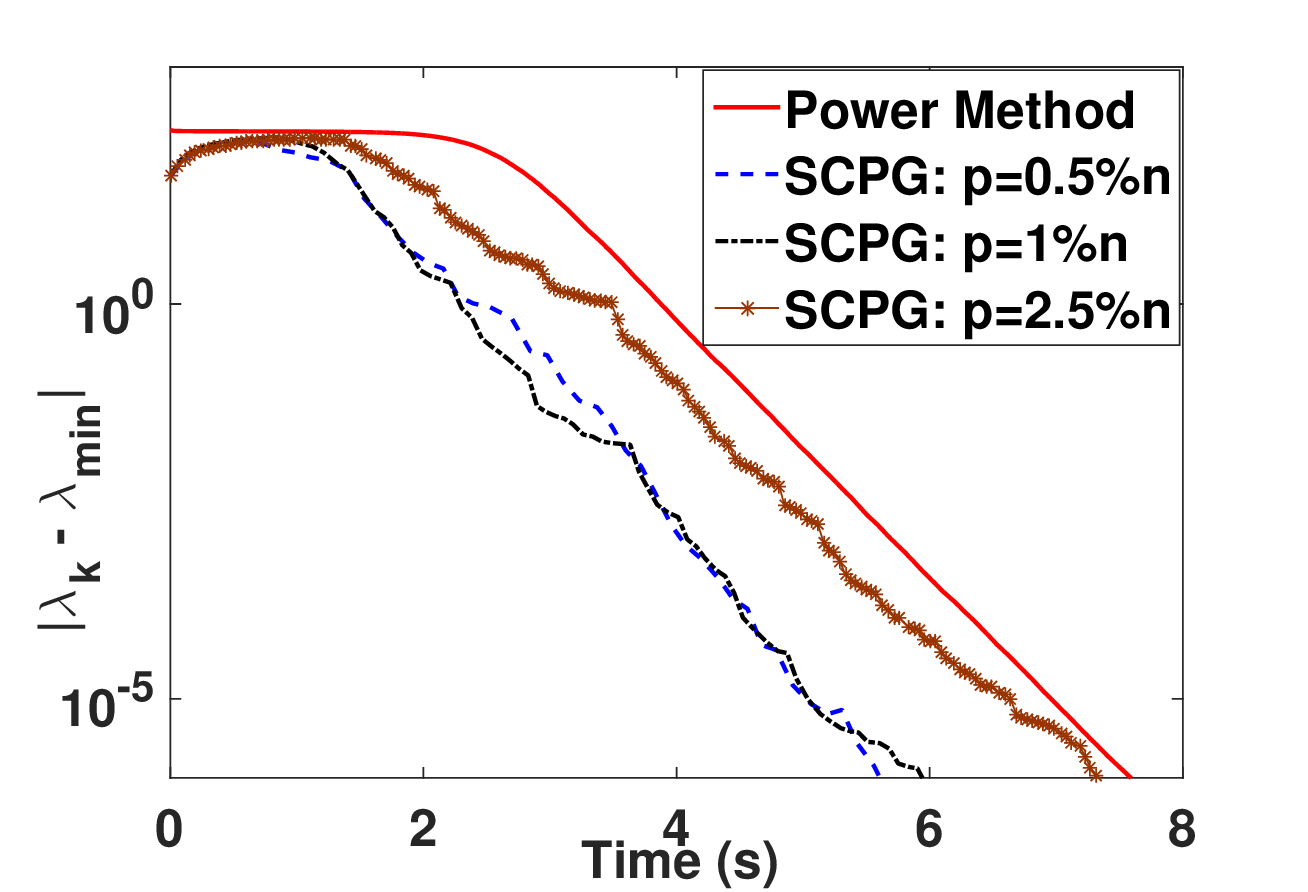} 
		\caption{Comparisson between SCPG and power method for the smalles eigenvalue on  group Schenk\underline{ }IBMNA.  Matrices: c-61 (left) n=43618  and c-65 (right) n=48066}
		\label{fig:2}
	\end{figure}

%%%%%%%%%%%%%%%%%%%%%%%%%%%%

\subsection{Logistic regression with cubic regularization}
\noindent Many problems from statistics, control and machine learning can be formulated as a finite sum  optimization problem with a regularization term \cite{Mit:97}.  For example, given a  dataset $(y_t,z_t)_{t=1}^T$, where $z_t \in \mathbb{R}^n$ is the vector of  features and $y_t \in \{-1,1\}$ is the label of the $t$ data, we consider the following classification  problem:
	\begin{align*}
		\min_{x \in \mathbb{R}^n}  F(x) := \frac{1}{T} \sum_{t=1}^T \ell(x,y_t,z_t) +  \frac{M}{6} \|x\|^3, 
	\end{align*}
	where  $\ell (\cdot)$ is the  loss function (e.g., the logistic loss $ \ell(x,y_t,z_t) =  \log (1 + e^{a_t^T x})$ for a given data $a_t  = y_t z_t$). Note that usually the loss is regularized in order to enforce unique solution but also to avoid overfitting \cite{Mit:97}.  In this section we  consider a cubic regularization term, $M/6 \|x\|^3$, which also forces uniqueness of the solution as this is a uniformly convex function.  It is known that the logistic regression function is smooth.  In this set of experiments we generate a set of sparse normally distributed  vectors  $(z_t)_{t=1}^T \subset \mathbb{R}^n$ and $(y_t)_{t=1}^T \subseteq \{-1, 1\}$ and we compare the algorithm  SCPG and the full gradient algorithm proposed in \cite{Nes:19}. As in the previous  simulations, we also choose the matrix $U_{k}$ as a  scaled sampling matrix and  stop the  algorithms  when  $\|\nabla F(x) \| \leq \epsilon = 10^{-2}$.  The number of full iterations (number of gradient evaluations $\nabla f$) and CPU time in seconds are given in Table 2. From table we observe that when the problem is ill-conditioned (i.e., $M$ small)  our algorithm is at least $2$ times faster w.r.t. time than the method in   \cite{Nes:19}. 
	
	\begin{table}[h!]
		\centering 
		\begin{center}
			\begin{tabular}{ |p{0.7cm}| p{1cm}| p{1.2cm}| p{1cm}| p{1cm}|}
				\hline
				\multirow{1}{*}{M} & \multicolumn{4}{|c|}{$n = 10^3$, $T = 10^3$, $p=20$} \\
				\cline{2-5}
				& \multicolumn{2}{|c|}{SCPG} & \multicolumn{2}{|c|}{\cite{Nes:19}}\\
				\hline
				1 &  17 &  \textbf{32.6} & 2540 &35.2\\
				\hline
				0.1 &  21 & \textbf{98.1} & 1125 & 118.2\\
				\hline
				0.01  & 34 & \textbf{201.4} & 8357 & 372.9\\
				\hline
				\hline
				M & \multicolumn{4}{|c|}{$n = 10^3$, $T = 10^2$, $p=20$} \\
				\hline
				1 & 16 & \textbf{4.1} & 2370 & 4.6\\
				\hline
				0.1 & 25 & \textbf{11.8} & 1071 & 17.2\\
				\hline
				0.01 & 48 & \textbf{22.1} & 5435 & 46.3\\
				\hline
				\hline
				M & \multicolumn{4}{|c|}{$n= 10^4$, $T = 10^3$, $p=125$} \\
				\cline{2-5}
				1 &  15 & \textbf{881.2} & 2431 & 896.2\\
				\hline
				0.1 & 22 & \textbf{1040.7} & 11930 & \text{1687.3}\\
				\hline
				0.01 & 48 & \textbf{1605.1} & 17005 & 4723.2\\
				\hline
				\hline
				M & \multicolumn{4}{|c|}{$n = 10^4$, $T = 10^2$, $p=125$} \\
				\hline
				1 & 17 & \textbf{597.1} & 3695 & 609.7\\
				\hline
				0.1  & 27 & \textbf{1190.7} & 13575 & 1986.5\\
				\hline
				0.01  & 48 & \textbf{1928.5} & 15486 & 5057.1\\
				\hline
			\end{tabular}
		\end{center}
		\caption{Number of full iterations and CPU time (in seconds) for  SCPG and  algorithm in \cite{Nes:19} on logistic regression.} 
	\end{table}

\subsection{Logistic regression with quadratic regularization} \red{We also consider the  classification  problem from previous section, but now replacing  the loss with  $\ell(x,y_{t},z_{t}) = - \left[y_{t} \log\left(1/(1+e^{-z_{t}^{T}x})\right) + (1-y_{t})\log\left(e^{-z_{t}^{T}x}/(1+e^{-z_{t}^{T}x})\right)    \right] = \log (1 + e^{z_{t}^{T}x}) - y_{t}z_{t}^{T}x $ and the regularization term   with $\lambda/2 \|x\|^2$. We consider the duke breast-cancer  dataset from \cite{data} with $n=7129$, $T=44$ and $y_t \in \{0,1\}$.  We solve this problem with the cubic Newton algorithm \cite{NesPol:06}, where at each iteration we solve the corresponding subproblem \eqref{opt_prob} with the following algorithms:  SCPG,  the algorithm  proposed in \cite{CarDuc:17} with $\eta = 1/(8(4\|A\| + 2MR))$ and the random coordinate gradient descent algorithm (RCGD) with Armijo line-search (Algorithm 2.1 in \cite{Bon:21} with $\beta = \delta_{i} = 0.5$). Note that the last two algorithms treat  $\psi(x)=M/6 \|x\|^3$ as part of $f(x) =  \langle b, x \rangle + \frac{1}{2}   \langle Ax, x \rangle$, i.e., at each iteration they compute a (partial)  gradient  of $F$,  $\nabla F(x) = \nabla f(x) + \nabla \psi(x)$.  Each algorithm is stopped when $\|\nabla F(x_{k})\| \leq 10^{-3}$.  The results in function values along time (sec)  for solving the logistic regression problem using  duke-breast-cancer dataset with cubic Newton method using  SCPG, algorithm in \cite{CarDuc:17} and RCGD with Armijo  line-search in \cite{Bon:21} as subroutines are given in Figure \ref{fig:3}. We consider two cases for the  regularization parameter $\lambda$, $0.01$ (left) and $0.0001$ (right),  and  three values for  $p$ (the number of coordinates updated in each iteration in SCPG/RCGD).  We see again that SCPG provides the best performance in terms of CPU time (sec).  
}

\begin{figure}[!ht]
		\centering
		\includegraphics[width=8cm, height=6.5cm]{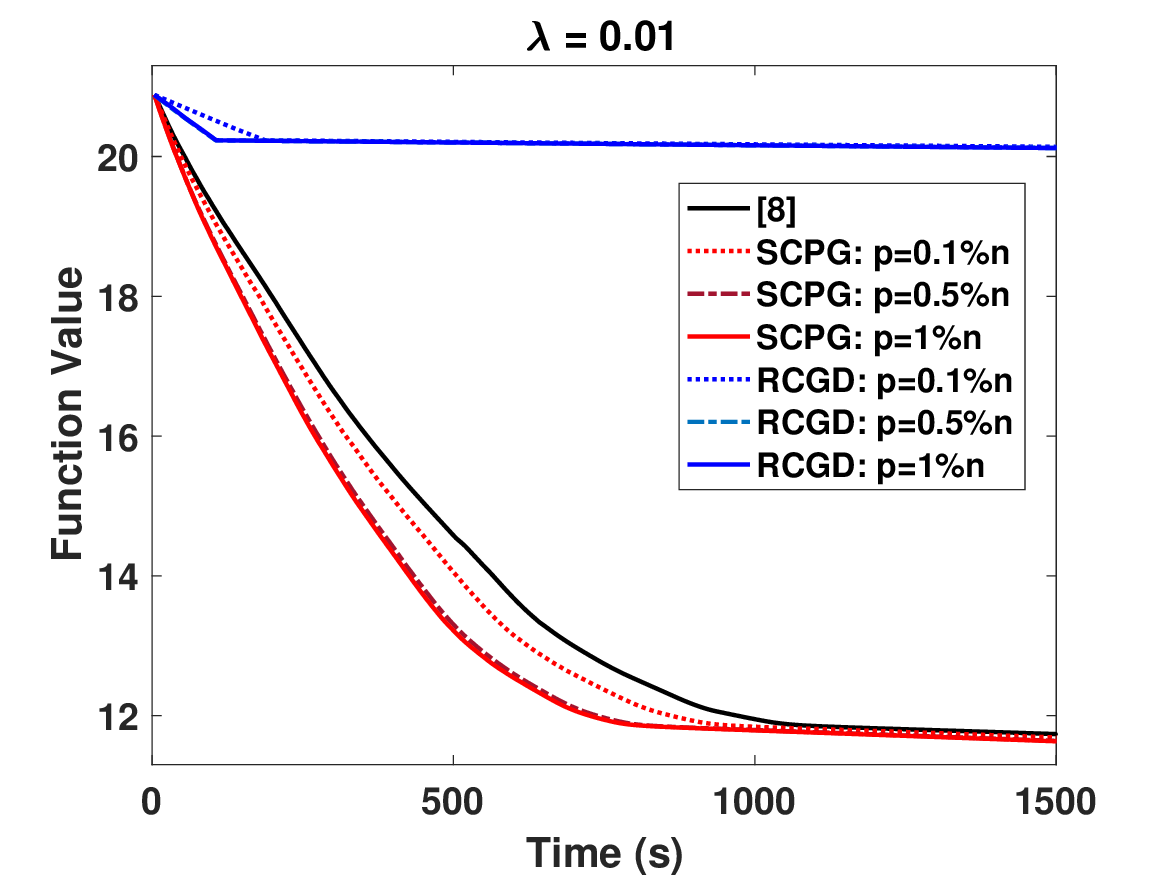}
       \includegraphics[width=8cm, height=6.5cm]{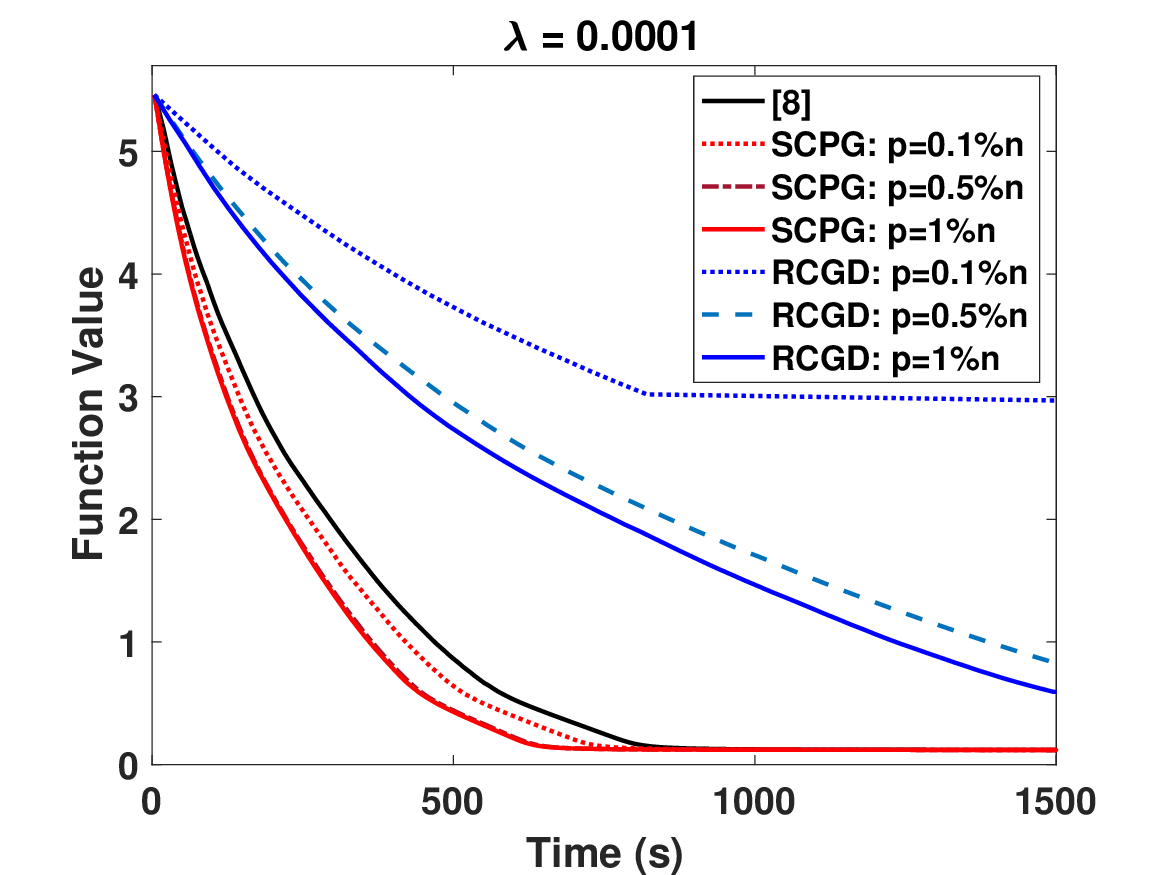}
		\caption{Solving logistic regression problem using  duke-breast-cancer dataset with cubic Newton using  SCPG, algorithm in \cite{CarDuc:17} and RCGD with Armijo  line-search in \cite{Bon:21} as subroutines: regularization parameter $\lambda$ is $0.01$ (left) and $0.0001$ (right), and three values for  $p$. }
		\label{fig:3}
\end{figure}

%%%%%%%%%%%%%%%%%%%	

\section{Conclusions}
In this paper we have designed stochastic coordinate proximal gradient methods for composite  optimization problems having  the  objective function formed as a sum of two  terms, one smooth and the second   possibly nonconvex and nonseparable.  We have provided probabilistic worst-case complexity analysis for our stochastic  coordinate descent method in convex and nonconvex settings, in particular we have proved high-probability bounds on the number of iterations before a given optimality is achieved.  To the best of our knowledge, this work is the first proposing a pure stochastic  coordinate descent algorithm which is supported by global efficiency estimates for general classes of nonseparable  composite optimization problems.   The   numerical results on random and real data  have also  confirmed the  efficiency of our  algorithm.

%%%%%%%%%%%%%%%%%%%%

 \begin{APPENDIX}{}
 	\label{apendix}
 	
 	\textit{Proof of Lemma \ref{lemma:1}}. From mean value theorem we have: 
 	\begin{equation*}
 		f(x+U h) - f(x) = \int_{0}^{1} \langle \nabla f(x+tUh), Uh \rangle dt.  
 	\end{equation*} 
 	\noindent If the function $f$ satisfies (\ref{lip1}), then it follows: 
 	\begin{align*}
 		& |f(x+U h) - f(x) - \langle \nabla f(x), Uh \rangle| = \left|\int_{0}^{1} \langle \nabla f(x+tUh) - \nabla f(x), Uh \rangle \, dt \right| \\
 		&\leq \int_{0}^{1} |\langle U^{T}(\nabla f(x+tUh) - \nabla f(x)), h \rangle| \, dt \\
 		& \leq \int_{0}^{1} \| U^{T}(\nabla f(x+tUh) - \nabla f(x))\| \|h\|  \, dt \\
 		&\leq  \int_{0}^{1} L_{U} \|th\| \|h\|  dt =  L_{U} \|h\|^{2} \int_{0}^{1} t \, dt = \frac{L_{U}}{2} \|h\|^2,
 	\end{align*}
 	which proves the statement. \Halmos
 	
 	%%%%%%%%%%%%%%%%%%%%%%%%%
 	
 	\medskip 
 	
 	\noindent \textit{Proof of Lemma \ref{lemma:2}}.  For any $x \in \mathbb{R}^{n}$ and $h \in \mathbb{R}^p$, we have: 
 	\begin{equation*}
 		U^{T} \int_{0}^{1} y \, dt = \int_{0}^{1} U^Ty \, dt  \qquad   \forall y \in \mathbb{R}^{n}.  
 	\end{equation*}
 	Therefore, we have:
 	\begin{equation*}
 		U^{T}(\nabla f(x + U h) - \nabla f(x)) = \int_{0}^{1} U^{T} \nabla^2 f(x+tUh)Uh \, dt. 
 	\end{equation*}	
 	\noindent If \eqref{eq:16} is satisfied, we get: 
 	\begin{align*}
 		& \|U^{T}(\nabla f(x + U h) - \nabla f(x))\| = \left\|\int_{0}^{1} U^{T} \nabla^2 f(x+tUh)Uh \, dt \right\| \\
 		& \leq \int_{0}^{1} \| U^{T} \nabla^2 f(x+tUh)Uh \| \, dt \leq  \int_{0}^{1} \| U^{T} \nabla^2 f(x+tUh)U\| \|h \| \, dt \\
 		&\leq \int_{0}^{1} L_{U} \|h \| \,  dt = L_{U} \|h \|.
 	\end{align*}
 	\noindent On the  other hand, if \eqref{lip1} is satisfied, then for any $x \in \mathbb{R}^{n}$, $h \in \mathbb{R}^{p}$ and $\alpha > 0$, we have:  
 	\begin{equation*}
 		\left\| \int_{0}^{\alpha} U^{T}\nabla^2f(x+tUh) Uh \, dt\right\| = \| U^{T}(\nabla f(x + \alpha U h) - \nabla f(x)) \| \leq \alpha L_{U} \|h\|.
 	\end{equation*}	
 	\noindent This implies that: 
 	\begin{equation*}
 		\left\| \dfrac{1}{\alpha} \int_{0}^{\alpha} U^{T}\nabla^2f(x+tUh) U \, dt \dfrac{h}{\|h\|}\right\| \leq L_{U}. 
 	\end{equation*}	
 	\noindent Taking $\alpha \to 0$, we get:  
 	\begin{equation*}
 		\left\| U^{T}\nabla^2f(x) U \dfrac{h}{\|h\|}\right\| \leq L_{U}. 
 	\end{equation*}	
 	\noindent From this and the definition of matrix norm  $\displaystyle \|A\| = \max_{\|h\| = 1} \|Ah\|$, we get \eqref{eq:16}.   \Halmos
 	
 	%%%%%%%%%%%%%%%%%%%%%%%%%%%%
 	
 	\medskip 
 	
 %	{\color{red}
 %		\textit{Proof of Lemma \ref{lemma:MVI}} Consider a vector $u \in \mathbb{R}^{m}$ and the parameterization $\alpha_{u}: [0,1] \to \mathbb{R}$, such that:
 %		\begin{equation*}
 %			\alpha_{u}(t) = \langle G\left(x+tUd\right) , u \rangle  
 %		\end{equation*}
 		
 %		\noindent By mean value theorem, there exist $\bar{t} \in [0,1]$, so that  
 %		\begin{equation*}
 %			\alpha_{u}(1) - \alpha_{u}(0) = \alpha_{u}'(\bar{t})  
 %		\end{equation*}
 		
 %		\noindent This implies that
 %		\begin{equation*}
 %			\langle G\left(x+Ud\right) - G\left(x\right)   , u \rangle  = \langle J\left( x+\bar{t}Ud\right)Ud, u \rangle 	
 %		\end{equation*}
 		
 %		\noindent Hence 
 %		\begin{equation*}
 %			|\langle G\left(x+Ud\right) - G\left(x\right)   , u \rangle| \leq \| J\left( x+\bar{t}Ud\right)Ud\| \|u\| \leq  \| J\left( x+\bar{t}Ud\right)U\| \|d\| \|u\|	
 %		\end{equation*}
 		
 %		\noindent Define $y = x+\bar{t}Ud$ and taking $u = \dfrac{G\left(x+Ud\right) - G\left(x\right)}{\|G\left(x+Ud\right) - G\left(x\right)\| } $, we get the statement.} 
 %	\Halmos
 	
 	%%%%%%%%%%%%%%%%%
 	
 		\textit{Proof of Lemma \ref{lem1}:}
 		First consider that $\psi$ is convex along the subspaces generated by matrices from $\mathcal{U}$. Using the optimality condition \eqref{optcond} for $d_{k}$ with the functions:
 	\begin{align*}
 		& \theta(d) = f(x_{k}) + \langle \nabla f(x_{k}), U_{k}d \rangle + \frac{\Hf }{2} \|d\|^2  \;  \text{and} \;   \phi(d) = \psi (x_{k} + U_{k}d), 
 	\end{align*}
 	we get:
 	\begin{align*}
 		%\label{oc1}
 		\langle U^{T}_{k} \nabla f(x_{k}) + \Hf d_{k}, d - d_{k} \rangle + \psi (x_{k} + U_{k}d)  \geq \psi (x_{k} + U_{k}d_{k}) \quad \forall d \in \mathbb{R}^{p}.
 	\end{align*}
 	From the previous  optimality condition for $d = 0$, we further obtain:
 	\begin{align*}
 		- \langle U^{T}_{k} \nabla f(x_{k}) + \Hf d_{k}, d_{k} \rangle + \psi(x_{k}) \geq \psi(x_{k} + U_{k}d_{k}) = \psi(x_{k+1}).
 	\end{align*}
 	Using Assumption \ref{ass2} and \eqref{lip2}, we obtain:
 	\begin{align}
 		f(x_{k+1}) + \psi(x_{k+1}) \leq f(x_{k}) + \langle \nabla f(x_{k}), U_{k}d_{k} \rangle + \frac{L_{U_{k}}}{2} \|d_{k}\|^2 + \psi(x_{k+1}). \label{eq:68}
 	\end{align}
 	Then, combining both inequalities, we get:
 	\begin{align*}
 		& F(x_{k+1})  = f(x_{k+1}) + \psi(x_{k+1})\\ 
 		& \leq \! f(x_{k}) + \langle \nabla f(x_{k}), U_{k}d_{k} \rangle + \frac{L_{U_{k}}}{2} \|d_{k}\|^2 \!-\! \langle U^{T}_{k} \nabla f(x_{k}), d_{k} \rangle  \!-\! \Hf \|d_{k}\|^2 + \psi(x_{k}) \\
 		&= F(x_{k}) - \left(\Hf - \dfrac{L_{U_{k}}}{2}\right)  \|d_{k}\|^2. 
 	\end{align*}
 	
 	\noindent \red{Similarly, if $\psi$ is a general function, given that $d_k$ is the global optimum of subproblem \eqref{eq:subproblem}, we have that the objective function in \eqref{eq:subproblem} evaluated in $d_k$ is smaller than in $0$, i.e.:}
 	\begin{equation*}
 		\langle U^{T}_{k} \nabla f(x_{k}), d_{k} \rangle + \dfrac{\Hf}{2} \|d_{k}\|^2  + \psi(x_{k} + U_{k}d_{k}) \leq  \psi(x_{k}). 
 	\end{equation*} 
 	
 	\noindent Combining this with inequality \eqref{eq:68}, we get:
 	\begin{eqnarray*}
 		F(x_{k+1}) &\leq& f(x_{k}) + \langle \nabla f(x_{k}), U_{k}d_{k} \rangle + \frac{L_{U_{k}}}{2} \|d_{k}\|^2 + \psi(x_{k+1}). \\ 
 		&\leq& F(x_{k}) - \dfrac{1}{2}\left(\Hf-L_{U_{k}}\right) \|d_{k}\|^2 .
 	\end{eqnarray*}	
 \noindent This proves our statement. \Halmos

 	%%%%%%%%%%%%%%%
 	
      \medskip  
 	
 	\textit{Proof of Lemma \ref{lem2}:}
\red{Since $f$ and $\psi$ are differentiable, then from the  optimality conditions for $d_{k}$, i.e., $d_k$ is a stationary point of subproblem \eqref{eq:subproblem},  we have that:}
\begin{align}
	\label{oc3}
	U^{T}_{k} (\nabla f(x_{k}) + \nabla \psi(x_{k+1})) +\Hf d_{k}  = 0. 
\end{align}

\noindent Moreover,   $\nabla F(x_{k}) = \nabla f(x_{k}) + \nabla \psi(x_{k})$. Using \eqref{oc3} and Assumptions \ref{ass2} and \ref{ass1},  we get  for all $k \in \mathbb{A}_{K}$:
\begin{eqnarray}
	\alpha \|\nabla F(x_{k})\| &\leq& \|U^{T}_{k} \nabla F(x_{k}) \| = \|U^{T}_{k}(\nabla f(x_{k}) + \nabla \psi(x_{k})) \| \nonumber \\
	&=& \|U^{T}_{k}(\nabla \psi(x_{k}) - \nabla \psi(x_{k+1})) - \Hf d_{k} \|. \nonumber 
\end{eqnarray}	

{	\noindent Consider $u_{j} \in \mathbb{R}^{n}$ the $j$th column of $U_{k}$ and $u_{ij} \in \mathbb{R}$  the $i$th element in the $j$th column. From the last inequality, we get:    
	\begin{eqnarray}
		\alpha^2 \|\nabla F(x_{k})\|^2 &\leq&  \sum_{j = 1}^{p} | u_{j}^{T}  (\nabla \psi(x_{k}) - \nabla \psi(x_{k+1})) - \Hf d_{k}^{(j)}|^2 \nonumber \\
		&=&  \sum_{j = 1}^{p} \left| \sum_{i = 1}^{n}  u_{ij}  (\nabla_{i} \psi(x_{k}) - \nabla_{i} \psi(x_{k+1})) - \Hf e_{j}^{T} d_{k} \right|^2,  \label{eq:56}
	\end{eqnarray}
	
	\noindent with $e_{j}$ the $j$th vector of the canonical basis. By the mean value theorem, we have that there exist $n$ vectors $z_{i} \in [x_{k},x_{k+1}]$ such that: 
	\begin{equation*}
		\nabla_{i} \psi(x_{k+1})-	\nabla_{i} \psi(x_{k}) = \nabla_{i}^2\psi(z_{i})U_{k} d_{k}  \quad \text{for} \quad i=1:n.
	\end{equation*}
	
	\noindent Using the equality above in \eqref{eq:56}, we have:  
	\begin{eqnarray}
		\alpha^2 \|\nabla F(x_{k})\|^2 &\leq&  \sum_{j = 1}^{p}  \left|\sum_{i = 1}^{n}  u_{ij} \nabla_{i}^2\psi(z_{i})U_{k} d_{k}  + \Hf e_{j}^{T} d_{k} \right|^2.  
		\label{eq:57}
	\end{eqnarray}
	
	\noindent Considering \eqref{hess} and the Frobenius norm definition, inequality \eqref{eq:57} becomes:
	\begin{eqnarray}
		\alpha^2 \|\nabla F(x_{k})\|^2 &\leq&  \sum_{j = 1}^{p}  | \left( u_{j}^{T} \bar{\nabla}^2\Psi(z_{1},\cdots,z_{n}) U_{k}  + \Hf e_{j}^{T}\right) d_{k}  |^2 \nonumber \\
		&\leq&  \sum_{j = 1}^{p}  \|u_{j}^{T} \bar{\nabla}^2\Psi(z_{1},\cdots,z_{n}) U_{k}   + \Hf e_{j}^T\|^2 \|d_{k} \|^2\nonumber \\
		&=&  \|U_{k}^{T}\bar{\nabla}^2\Psi(z_{1},\cdots,z_{n})U_{k} + \Hf I_{p\times p} \|_{F}^{2} \, \|d_{k} \|^2.  \nonumber
\end{eqnarray}}

\begin{comment}
{\color{red}\noindent Since, by assumption $\mathcal{L}_{F}(x_{0})$ is bounded,  then its convex hull  $\text{conv}(\mathcal{L}_{F}(x_{0}))$ is also bounded. Using the fact that $\psi$ is twice continuously  differentiable, from Lemma \ref{lemma:MVI} and Remark \ref{rem:MVI}, we have that exists $\Hpsimax > 0$ such that $\|U_{k}^{T} \nabla^2 \psi (x) U_{k} \| \leq \Hpsimax$, for all $x\in \text{conv}(\mathcal{L}_{F}(x_{0}))$ and $U \sim  \mathcal{U}$. Moreover  	
%	and $\mathcal{U}$ is bounded, we have that there exist $\Hpsibar < \infty$ such that $\|U^{T}_{k}\nabla^2 \psi (\bar{x}_k)\| \leq \Hpsibar$ for all $\bar{x}_k \in \text{conv}(\mathcal{L}_{F}(x_{0}))$ and $U_{k} \sim \mathcal{U}$. Moreover, by Lemma \ref{lem1}, $x_{k}, x_{k} + U_{k}d_{k} \in  \mathcal{L}_{F}(x_{0})$, this follows $ y_{k} \in \text{conv}(\mathcal{L}_{F}(x_{0}))$, for all $y_{k} \in (x_{k}, x_{k} + U_{k}d_{k})$. Hence $\|U^{T}_{k}\nabla^2 \psi (y_k)\| \leq \Hpsibar$ for all $y_{k} \in (x_{k}, x_{k} + U_{k}d_{k})$. So using the mean value inequality, we get:
%\begin{align}
%	\|U^{T}_{k}(\nabla \psi(x_{k} + U_{k}d_{k}) - \nabla \psi(x_{k}))\| 
%	 \leq \Hpsibar \|U_{k} d_{k}\| \leq \Hpsibar \|U_{k}\| \|d_{k}\| .   \label{eq:mvi}
%\end{align}
%\noindent  Then, from the boundness of $\mathcal{U}$, \eqref{Hpsimax} and  \eqref{eq:mvi}, we obtain: 
\begin{align}
\|U^{T}_{k}(\nabla \psi(x_{k} + U_{k}d_{k}) - \nabla \psi(x_{k}))\| 
\leq \Hpsimax \|d_{k}\|. \label{eq:100}
\end{align}
\end{comment}
\noindent Using \eqref{Hpsi} the  statement follows. \Halmos
 	
 	%%%%%%%%%%%%%%%%
 	
 	\medskip 
 	
 		\textit{Proof of Lemma \ref{lem3}:}
 		Let $T_{i}$ be the indicator function for the event: \textit{$U_{i}$ is well-aligned}. Then,  $|\mathbb{A}_{K}| = \sum_{i = 0}^{K} T_{i}$. Since $T_{i} \in \{0,1\}$, denoting  $p_{i} := \mathbb{P}[T_{i} = 1|x_{i}]$, then for any $t>0$ and  $i=0, 1, \cdots,K$, we have:  
 		\begin{eqnarray*}
 			\mathbb{E}\left[e^{-t(T_{i} - p_{i})}|x_{i}\right] &=& p_{i} e^{-t(1-p_{i})} + (1-p_{i})e^{tp_{i}}  = e^{tp_{i}}(1-p_{i}+p_{i}e^{-t}) \\
 			&=& e^{tp_{i}}e^{\log(1-p_{i}+p_{i}e^{-t})} =  e^{tp_{i} + \log(1-p_{i}+p_{i}e^{-t})} \leq e^{\frac{t^2p_{i}}{2}}, \label{eq:2}
 		\end{eqnarray*}
 		where the inequality follows from the relation: 
 		\begin{align}
 			\label{eq:ineq}
 			py + \log(1-p + pe^{-y}) \leq \frac{y^2p}{2} \quad  \forall p \in [0,1], \; y \geq 0.
 		\end{align}  
 
 \noindent Indeed, let us denote $\zeta(x) = px + \log(1-p + pe^{-x})$ and $\eta(x) = \dfrac{px^{2}}{2}$. We have $\zeta(0) = \eta(0) = 0$ and also the relation:  
 	\begin{equation*}
 		\zeta'(x) = p - \dfrac{pe^{-x}}{1-p+pe^{-x}}, \;\;  \eta'(x) = px. 
 	\end{equation*}
 	
 	\noindent Hence $\zeta'(0) = \eta'(0) = 0$. Moreover, $\eta''(x) = p$ and the relation:  
 	\begin{equation}
 		\zeta''(x) = \dfrac{(1-p)pe^{-x}}{(1-p + pe^{-x})^{2}}. 
 		\label{eq:8}
 	\end{equation}
 	
 	\noindent We further  have:  
 	\begin{align*}
 		&  0 \leq \left( (1-p) - e^{-x}\right)^2 \iff 0 \leq (1-p)^2 - 2(1-p)e^{-x} + (e^{-x})^2, \\
 		& \iff 4(1-p)e^{-x} \leq (1-p)^2 + 2(1-p)e^{-x} + (e^{-x})^2 = \left( 1-p + e^{-x}\right)^2, \\
 		& \iff 2 \sqrt{(1-p)e^{-x}} \leq 1-p + e^{-x}, \\
 		& \iff 2 \sqrt{(1-p)e^{-x}} - (1-p)e^{-x}  \leq 1-p + pe^{-x}.
 	\end{align*}
 	
 	\noindent Using this in \eqref{eq:8}, we obtain:
 	\begin{equation*}
 		\zeta''(x) \leq \dfrac{(1-p)pe^{-x}}{\left(2 \sqrt{(1-p)e^{-x}} - (1-p)e^{-x}\right)^2}.
 	\end{equation*}
 	
 	\noindent Since $x \geq 0$ and  $0\leq p \leq 1$, we have $(1-p)e^{-x} \geq 0$ and thus: 
 	\begin{equation*}
 		\sqrt{(1-p)e^{-x}} \leq (1-p)e^{-x} \iff \dfrac{1}{(1-p)e^{-x}} \geq \dfrac{1}{\sqrt{(1-p)e^{-x}}}.
 	\end{equation*} 
 	
 	\noindent Finally, we get:
 	\begin{eqnarray*}
 		\zeta''(x) & \leq &  \dfrac{p(1-p)e^{-x}}{\left(2 \sqrt{(1-p)e^{-x}} - \sqrt{(1-p)e^{-x}}\right)^2} \\ 
 		&=&  \dfrac{p(1-p)e^{-x}}{\left(\sqrt{(1-p)e^{-x}}\right)^2} = p = \eta''(x).
 	\end{eqnarray*}
 	
 	\noindent  Since $\eta''(x) - \zeta''(x) \geq 0$ for all  $x \geq 0$, it follows that  the function $h(x) = \eta(x) - \zeta(x)$ is convex  in the set  $\mathbb{B} = \{x: \;  x \geq 0\}$.  Moreover, since $h'(0) = 0$, then $0$ is a minimizer of $h$ over  $\mathbb{B}$. Since $h(0) = 0$, we have $0 \leq  h(x)$ for $x \in \mathbb{B}$. Hence,   $\zeta(x) \leq \eta(x)$, i.e., \eqref{eq:ineq} holds.
 	 Using basic properties of conditional expectations and the fact that $T_{i}$ only depends on $x_{i}$ and not on the previous iterations (see  Assumptions  \ref{ass1}), we then get: 
 		\begin{eqnarray}
 			\mathbb{E}\left[e^{-t\left(|\mathbb{A}_{K}| - \sum_{i=0}^{K} p_{i}\right) }\right] &=& \mathbb{E}\left[e^{-t\sum_{i=0}^{K}(T_{i}-p_{i})}\right] \nonumber \\
 			&=&  \mathbb{E} \left[e^{-t\sum_{i=0}^{K-1}(T_{i}-p_{i})} \mathbb{E}\left[e^{-t(T_{K} - p_{K})}|x_{K}\right] \right] \nonumber \\
 			&\leq& e^\frac{{t^2p_{K}}}{2} \mathbb{E}\left[e^{-t \sum_{i=0}^{K-1}(T_{i} - p_{i})}\right] \label{eq:3}  \\
 			&=& e^\frac{{t^2p_{K}}}{2} \mathbb{E}\left[e^{-t \sum_{i=0}^{K-2}(T_{i} - p_{i})} \mathbb{E}\left[e^{-t(T_{K-1} - p_{K-1})}|x_{K-1}\right]\right] \nonumber \\
 			&\leq& e^\frac{{t^2p_{K}}}{2} e^\frac{{t^2p_{K-1}}}{2}  \mathbb{E}\left[e^{-t \sum_{i=0}^{K-2}(T_{i} - p_{i})}\right] \leq e^{\frac{t^2}{2} \left(\sum_{i=0}^{K}p_{i}\right)}. \nonumber  
 		\end{eqnarray} 
 		
 		\noindent By Markov inequality and \eqref{eq:3}, we have:
 		\begin{eqnarray}
 			\mathbb{P}\left[e^{-t\left(|\mathbb{A}_{K}| - \sum_{i=0}^{K} p_{i}\right)} \geq e^{t \lambda}\right] &\leq& e^{-t\lambda} \cdot  \mathbb{E}\left[e^{-t\left(|\mathbb{A}_{K}| - \sum_{i=0}^{K} p_{i}\right) }\right] \nonumber \\
 			&\leq& e^{-t\lambda} \cdot e^{\frac{t^2}{2} \left(\sum_{i=0}^{K}p_{i}\right)} = e^{\frac{t^2}{2} \left(\sum_{i=0}^{K}p_{i}\right)-t\lambda}. \label{eq:4}
 		\end{eqnarray}
 		
 		\noindent On the  other hand,  one can  note that:  
 		\begin{eqnarray*}
 			|\mathbb{A}_{K}| \leq \sum_{i=0}^{K} p_{i} - \lambda & \iff  & t\lambda \leq -t|\mathbb{A}_{K}| + t \sum_{i=0}^{K} p_{i}  \iff  e^{-t\left( |\mathbb{A}_{K}| - \sum_{i=0}^{K} p_{i}\right) } \geq e^{t \lambda}. 
 		\end{eqnarray*}
 		
 		\noindent Hence, we get: 
 		\begin{eqnarray}
 			\mathbb{P}\left[|\mathbb{A}_{K}| \leq \sum_{i=0}^{K} p_{i} - \lambda\right] = \mathbb{P}\left[e^{-t\left(|\mathbb{A}_{K}| - \sum_{i=0}^{K} p_{i}\right)} \geq e^{t \lambda}\right]. \label{eq:5}
 		\end{eqnarray}
 		
 		\noindent Combining \eqref{eq:4} and \eqref{eq:5}, we further get: 
 		\begin{equation*}
 			\mathbb{P}\left[|\mathbb{A}_{K}| \leq \sum_{i=0}^{K} p_{i} - \lambda\right] \leq e^{\frac{t^2}{2} \left(\sum_{i=0}^{K}p_{i}\right)-t\lambda}.
 		\end{equation*}
 		
 		\noindent Taking $\displaystyle t = \frac{\lambda}{\sum_{i=0}^{K}p_{i}}$, we obtain:
 		\begin{equation*}
 			\mathbb{P}\left[|\mathbb{A}_{K}| \leq \sum_{i=0}^{K} p_{i} - \lambda\right] \leq e^{-\frac{\lambda^2}{2 \left( \sum_{i=0}^{K}p_{i}\right) }}.
 		\end{equation*}
 		
 		\noindent Finally,  taking $\displaystyle \lambda = \beta \sum_{i=0}^{K} p_{i}$ for some $\beta \in (0,1)$, we have:
 		\begin{equation*}
 			\mathbb{P}\left[|\mathbb{A}_{K}| \leq (1- \beta)\sum_{i=0}^{K} p_{i} \right] \leq e^{-\frac{\beta^2\left( \sum_{i=0}^{K}p_{i}\right) }{2}}. 
 		\end{equation*}
 		
 		\noindent and consequently 
 		\begin{equation}
 			\mathbb{P}\left[|\mathbb{A}_{K}| \geq (1- \beta)\sum_{i=0}^{K} p_{i} \right] \geq 1 - e^{-\frac{\beta^2\left( \sum_{i=0}^{K}p_{i}\right) }{2}}. \label{eq:6}
 		\end{equation}
 		
 		\noindent From Assumption \ref{ass1},  $p_{i} \geq (1-\delta)$ for all $i=0, 1, \cdots,K$. Therefore, we get: 
 		\begin{equation*}
 			(1-\beta) \sum_{i=0}^{K} p_{i} \geq (1-\beta)(1-\delta)(K+1).
 		\end{equation*}
 		
 		\noindent Since the event $\{|\mathbb{A}_{K}| \geq (1- \beta)\sum_{i=0}^{K} p_{i}\}$  is contained in the event $\{|\mathbb{A}_{K}| \geq (1-\beta)(1-\delta)(K+1)\}$, from  \eqref{eq:6} we obtain: 
 		\begin{align*}
 			\mathbb{P}\left[|\mathbb{A}_{K}| \geq (1-\beta)(1-\delta)(K+1)\right]  \geq \mathbb{P}\left[|\mathbb{A}_{K}| \geq (1- \beta)\sum_{i=0}^{K} p_{i} \right]  \geq 1-e^{-\frac{\beta^2\left( \sum_{i=0}^{K}p_{i}\right) }{2}}.
 		\end{align*}
 		
 		\noindent Since we also have $\displaystyle \sum_{i=0}^{K} p_{i} \geq (1-\delta)(K+1)$, we get our statement \eqref{eq:7}. 
 		\Halmos
 		
 		%%%%%%%%%%%%%%%%
 		
 		\medskip

 		\textit{Proof of Lemma \ref{LemmaKL}:}
 			By Lemma \ref{lem1}, we have $F(x_{k}) \leq F(x_{0})$ for all $k \geq 0$.  Since $\mathcal{L}_{f}(x_{0})$ is assumed bounded, then the  sequence $(x_{k})_{k\geq 0}$ is also bounded. This implies that the set $\mathcal{X}(x_{0})$ is also bounded. Closeness of $\mathcal{X}(x_{0})$ also follows observing that $\mathcal{X}(x_{0})$ can be viewed as an intersection of closed sets, i.e., $\mathcal{X}(x_{0}) = \cap_{j \geq 0} \cup_{\ell \geq j} \{x_{k_{\ell}}\}$. Hence $\mathcal{X}(x_{0})$  is a compact set.   Further, using the boundness of $(x_{k})_{k\geq 0}$ and the continuity  of $F$ and $\nabla F$, we have that the sequences $\left( F(x_{k})\right) _{k\geq 0}$ and $\left( \| \nabla F(x_{k})\|^2 \right) _{k\geq 0}$ are also bounded.   From Lemma \ref{lem1}, we have: 
 		\begin{equation*}
 			\sum_{k=0}^{K} \|d_{k}\|^{2} < \infty \quad  \text{a.s.},   \;\; \text{hence} \;  \;  \| d_k \|  \to 0 \quad  \text{a.s.} 
 		\end{equation*}
 		Hence, by Lemma \ref{lem2} we have:
 		\begin{equation}
 			\|\nabla F(x_{k_{j}})\|  \to  0 \quad  \text{a.s.}   \label{eq:58}
 		\end{equation}
 		
 		\noindent Moreover from Lemma \ref{lem1} we have that $F(x_{k})_{k\geq 0}$ is monotonically decreasing and since $F$ is assumed bounded from below by $F^{*} > - \infty$, it converges, let us say to $F_{*} > - \infty$, i.e. $F(x_{k}) \to F_{*}$ a.s., and   $F_{*}  \geq  F^{*}$.
 		On the other hand, let $\bar{x}$ be a limit point of $(x_{k_{j}})_{j\geq 0}$, i.e. $\bar{x} \in \mathcal{X}(x_{0})$. This means that there is a subsequence $(x_{\bar{k}})_{\bar{k}\geq 0}$ of $(x_{k_{j}})_{j\geq 0}$
 		such that $x_{\bar{k}} \to \bar{x}$ a.s.  Since  $f$ and $\psi$ are continuous, it follows that $F(x_{\bar{k}}) \inas F(\bar{x})$. This implies that $F(\bar{x}) = F_{*}$. Finally,  to prove that $\nabla F(\mathcal{X}(x_{0})) = 0$, one can notice that since $x_{\bar{k}} \to \bar{x}$, then  $\nabla F(x_{\bar{k}}) \to \nabla F(\bar{x})$ a.s.  Using basic probabilistic arguments,  by \eqref{eq:58}, we get $\|\nabla F(\bar{x})\| = 0$ a.s. 
 		\Halmos

 		%%%%%%%%%%%%%%%%%%%
 		
        \medskip 
        	
 	\textit{Proof of Lemma \ref{lemma:rec}:}   For the  recurrence $\Delta_{k} - \Delta_{k+1} \geq \Delta_{k+1}^{\zeta+1}$, with $\zeta \in (0,1]$, a sublinear bound was derived in \cite{Nes:19}. We extend this result to more general $\zeta$ and to a  different recurrence  in the form \eqref{eq:78}.  Consider $c=1$ and  $\zeta > 0$. Multiplying \eqref{eq:78} by $ \Delta_{k}^{-(\zeta+1)} $ we get: 
 	\begin{equation}
 		1 \leq \left( \Delta_{k}^{-\zeta} - \dfrac{\Delta_{k+1}}{\Delta_{k}^{1+\zeta}} \right) = \left( 1 - \dfrac{\Delta_{k+1}}{\Delta_{k}}\right) \Delta_{k}^{-\zeta}. \label{eq:82}
 	\end{equation}
 	
 	\noindent Let us show that the function $g(y) = y^{-\zeta} + \zeta y - (1+\zeta)$ satisfies  $g(y) \geq 0$ for all $y \in (0,1]$ and $\zeta>0$. Indeed, since: 
 	\begin{equation*}
 		g'(y)  = -\zeta y^{-(\zeta + 1)} + \zeta \leq 0  \quad  \forall y \in (0,1],  
 	\end{equation*}
 	it follows that $g(y)$ is decreasing on $(0,1]$. Moreover, $g(1) = 0$ and also
 	\begin{equation*}
 		\lim_{y \searrow 0} g(y) = \lim_{y \searrow 0} \left(\dfrac{1}{y^{\zeta}} + \zeta y - (1-\zeta)\right) = + \infty.   %\label{eq:81}
 	\end{equation*}
 	
 	\noindent Therefore,  $g(y) \geq 0$ for all $y \in (0,1]$ and $\zeta>0$. Since $\dfrac{\Delta_{k+1}}{\Delta_{k}} \in (0,1]$, we get:
 	\begin{eqnarray*}
 		\left( \dfrac{\Delta_{k+1}}{\Delta_{k}}\right) ^{-\zeta} + \zeta \left( \dfrac{\Delta_{k+1}}{\Delta_{k}}\right)  -1-\zeta \geq 0, 
 	\end{eqnarray*}
 	or equivalently
 	\begin{eqnarray*}
 		\dfrac{1}{\zeta}\left[\left(\dfrac{\Delta_{k+1}}{\Delta_{k}}\right)^{-\zeta} - 1 \right] \geq \left( 1 - \dfrac{\Delta_{k+1}}{\Delta_{k}}\right). 
 	\end{eqnarray*}
 	\noindent Combining the last inequality with  \eqref{eq:82}, we obtain: 
 	\begin{equation*}
 		1 \leq \dfrac{\Delta_{k}^{-\zeta}}{\zeta}\left[ \left(\dfrac{\Delta_{k+1}}{\Delta_{k}}\right)^{-\zeta} - 1 \right] = \dfrac{\Delta_{k+1}^{-\zeta} - \Delta_{k}^{-\zeta}}{\zeta}. 
 	\end{equation*}
 	\noindent Summing this relation, we get:
 	\begin{equation*}
 		k\zeta \leq \sum_{i = 0}^{k-1} \left( \Delta_{i+1}^{-\zeta} - \Delta_{i}^{-\zeta} \right)  = \Delta_{k}^{-\zeta} -  \Delta_{0}^{-\zeta},  
 	\end{equation*}
 	and by arranging the terms we finally obtain:  
 	\begin{equation*}
 		\Delta_{k}^{\zeta} \leq  \dfrac{\Delta_{0}^{\zeta}}{\zeta k\Delta_{0}^{\zeta} + 1}, 
 	\end{equation*}
 	thus proving  \eqref{eq:61}.  If $\zeta=0$ and $c\in(0,1)$, we get:
 	\begin{align*}
 		&  \Delta_{k+1} \leq \left(1 - c \right) \Delta_{k},
 	\end{align*}	
 	\noindent thus proving \eqref{eq:62}.  Finally, for  $c>0$ and $-1 < \zeta < 0$, we have $ \Delta_{k+1} \leq  \Delta_{k}$ and consequently \eqref{eq:78} leads to the following recurrence: 
 	\begin{equation*}
 		\Delta_{k} - \Delta_{k+1} \geq c \Delta_{k+1}^{\zeta+1} \quad \forall k \geq 0. 
 	\end{equation*}
 	Rearranging the terms we get: 
 	\begin{equation*}
 		\Delta_{k+1}\left(1 + c \Delta_{k+1}^{\zeta}\right) \leq \Delta_{k} \iff \Delta_{k+1} \leq \left(\dfrac{1}{1 + c \Delta_{k+1}^{\zeta}}\right) \Delta_{k}, %\label{eq:65}
 	\end{equation*}
 	which yields 	\eqref{eq:63}. Note that $\Delta_{k+1}^\zeta \to \infty$, provided that $\zeta<0$.  These prove  our statements. \Halmos
 
\end{APPENDIX}

%%%%%%%%%%%%%%%%%%%%%

\section*{Acknowledgments.}
\noindent The research leading to these results has received funding from: ITN-ETN project TraDE-OPT funded by the European Union’s Horizon 2020 Research and Innovation Programme under the Marie Skłodowska-Curie grant agreement No.  861137; NO Grants 2014-2021, under project ELO-Hyp,  no. 24/2020;  UEFISCDI PN-III-P4-PCE-2021-0720, under project L2O-MOC, nr.  70/2022.

%%%%%%%%%%%%%%%%%%%%%%

\end{document}